\pdfoutput=1
\documentclass[12pt,reqno]{article}

\usepackage[usenames]{color}
\usepackage{amssymb}
\usepackage{amsmath}
\usepackage{amsthm}
\usepackage{amsfonts}
\usepackage{amscd}
\usepackage{graphicx}
\usepackage{adjustbox}
\usepackage{enumitem}
\usepackage[colorlinks=true,
linkcolor=webgreen,
filecolor=webbrown,
citecolor=webgreen]{hyperref}
\usepackage[ruled]{algorithm2e}
\SetKwIF{If}{ElseIf}{Else}{if}{}{else if}{else}{end if}%
\usepackage{booktabs}
\usepackage{multicol}
\usepackage{siunitx}
\definecolor{webgreen}{rgb}{0,.5,0}
\definecolor{webbrown}{rgb}{.6,0,0}

\usepackage{color}
\usepackage{fullpage}
\usepackage{float}
\usepackage{listings}
\usepackage{xcolor}
\usepackage{graphics}
\usepackage{latexsym}
\usepackage{epsf}
\usepackage{breakurl}
\usepackage{multirow}
\usepackage{caption}
\usepackage{subcaption}
\usepackage[noabbrev]{cleveref}
\usepackage{microtype}
\usepackage{mathtools}
\makeatletter
\newcommand{\customsize}{\@setfontsize\customsize{5}{7}}
\makeatother
\usepackage{booktabs,makecell}
\usepackage{siunitx}

\def\Enn{\mathbb{N}}
\def\Zee{\mathbb{Z}}
\def\Que{\mathbb{Q}}

\newcommand{\limp}{\mathbin\rightarrow}

\makeatletter
\g@addto@macro\bfseries{\boldmath}
\makeatother

\newcommand{\CaDiCaL}{\textsc{CaDiCaL}}
\newcommand{\march}{\textsc{march}}
\newcommand{\0}{{\tt0}}
\newcommand{\1}{{\tt1}}

\begin{document}

\theoremstyle{plain}
\newtheorem{theorem}{Theorem}
\newtheorem{corollary}[theorem]{Corollary}
\newtheorem{lemma}[theorem]{Lemma}
\newtheorem{proposition}[theorem]{Proposition}

\newtheorem{definition}[theorem]{Definition}
\theoremstyle{definition}
\newtheorem{example}[theorem]{Example}
\newtheorem{conjecture}[theorem]{Conjecture}

\theoremstyle{remark}
\newtheorem{remark}[theorem]{Remark}
\title{North--East Lattice Paths Avoiding $k$ Collinear Points via Satisfiability}
\author{
Aaron Barnoff \\
School of Computer Science \\
University of Windsor \\
Canada \\
\texttt{barnoffa@uwindsor.ca} \\
\and
Curtis Bright \\
School of Computer Science \\
University of Waterloo \\
Canada \\
\texttt{cbright@uwaterloo.ca} \\
}
\date{November 27, 2025}
\maketitle

\begin{abstract}
We investigate the Gerver--Ramsey collinearity problem
of determining the maximum number of points in a north--east lattice path without $k$ collinear points.
Using a satisfiability solver, up to isomorphism we enumerate all north--east lattice paths
avoiding~$k$ collinear points for $k\leq6$.  We also find a north--east lattice path
avoiding $k=7$ collinear points with 327 steps, improving on the previous best length of 260 steps found by Shallit.
\end{abstract}

\section{Introduction}

In 1971, Tom C.~Brown~\cite{Brown1971} asked the following: must every sufficiently long lattice path in the plane
with steps in $\{(1,0),(0,1)\}$ always contain $k$ collinear points, regardless of
the value of~$k\geq1$?
The following year, P.~L.~Montgomery~\cite{Montgomery1972} published a solution showing the answer to be yes:
every sufficiently long north--east lattice path must contain $k$ collinear points,
regardless of the choice of~$k$.

However, Montgomery did not provide a constructive bound on how long the walk
must be before $k$ collinear points were guaranteed.
In 1979, Gerver and Ramsey~\cite{GerverRamsey1979} provided such a bound.
They showed every north--east lattice path in the plane of length at least
\begin{equation}
 (k-1) 2^{2^{13}(k-1)^4} \label{eq:upper}
\end{equation}
must contain~$k$ collinear points.
The Gerver--Ramsey bound, while explicit, is extremely loose.
For example, the bound guarantees that every north--east walk of length at least~$2^{131073}$ contains $k=3$ collinear points,
although in fact every walk with just four steps contains three collinear points (see Figure~\ref{fig:4steps}).
In a separate paper published at the same time as Gerver and Ramsey's bound,
Gerver~\cite{GerverPaper} showed that there exists a north--east walk of length greater than
\begin{equation}
(32 (k-1)^{2\log_2(k-1)-7})^{1/18} \label{eq:lower}
\end{equation}
avoiding $k$ collinear points.  For example, for $k=3$, this bound says there exists a walk
with more than $(32\cdot2^{-5})^{1/18} = 1$ step avoiding three collinear points.
Although this is a super-polynomial bound, it is quite loose for small values of $k$:
it does not guarantee the existence of a 3-step walk avoiding $k$
collinear points until $k=30$.

The large gap between~\eqref{eq:upper} and~\eqref{eq:lower} means that
precisely how long north--east lattice paths can be while avoiding
$k$ collinear points is unknown.  In this paper, we study the problem of
determining this length exactly for small values of~$k$---%
a problem posed in 1979 by A.~Meir~\cite{crux}.
Let $a(k)$ denote the smallest integer
such that all north--east lattice paths of length $a(k)$ contain $k$ collinear points,
so that $a(k)-1$ is the length of the longest north--east lattice path
avoiding $k$ collinear points.
Meir noted that $a(3)=4$~\cite{crux},
and in 2013, J.~Shallit~\cite{A231255} computationally determined
$a(4)=9$, $a(5)=29$, and $a(6)=97$.
He also
established the lower bound $a(7)\ge 261$
by finding a north--east lattice path of length 260 without 7 collinear points.

We call a north--east lattice path without $k$ collinear points a \emph{$\textit{GR}(k)$ walk}
in honour of Gerver and Ramsey, and we call a point \emph{GR$(k)$-reachable} if it is
reachable from the origin by a GR($k$) walk.  In addition, if the GR($k$) walk has length $a(k)-1$ (i.e.,
has $a(k)-1$ steps and therefore contains $a(k)$ points) we call the GR($k$) walk \emph{maximal}.
Although the results of Montgomery and Gerver--Ramsey imply a maximal GR($k$) walk
exists for every $k$, there may exist multiple maximal GR($k$) walks for given $k$.
Up to isomorphism, we find there are two distinct maximal GR(4) and GR(6) walks
and a single maximal GR(5) walk (see Section~\ref{sec:enumeration-kleq6}).
The unique maximal GR(5) walk is visually depicted in Figure~\ref{fig:28steps}.

\begin{figure}
    \centering
    \includegraphics[width=0.6\linewidth]{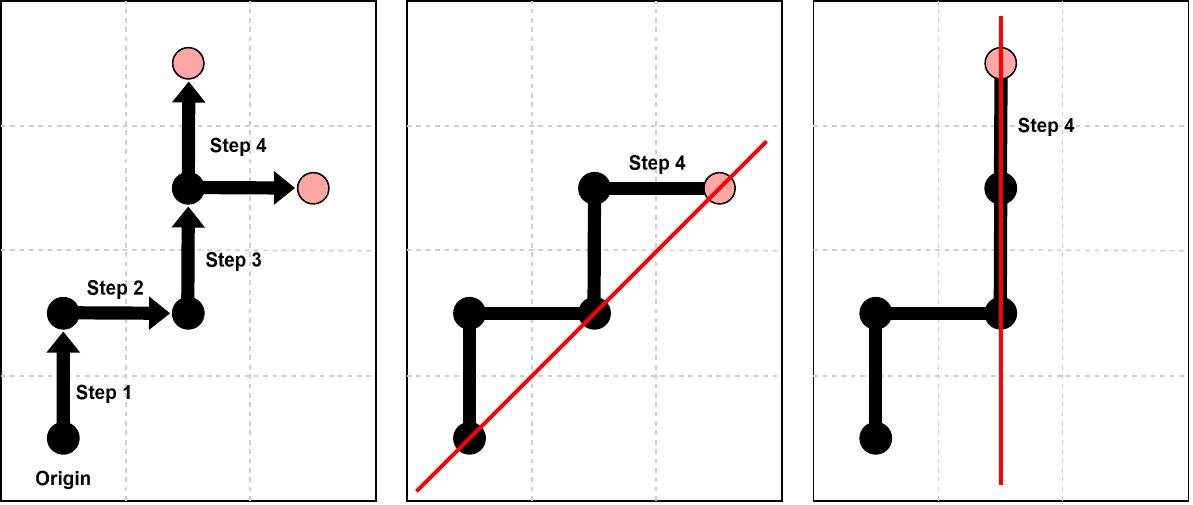}
    \caption{A visual representation of the longest north--east walk avoiding $k=3$ collinear points.  Walking two steps in the same direction
    introduces three collinear points, so the longest walk avoiding three collinear points alternates directions.}
    \label{fig:4steps}
\end{figure}

\begin{figure}
    \centering
    \includegraphics[width=0.3\linewidth]{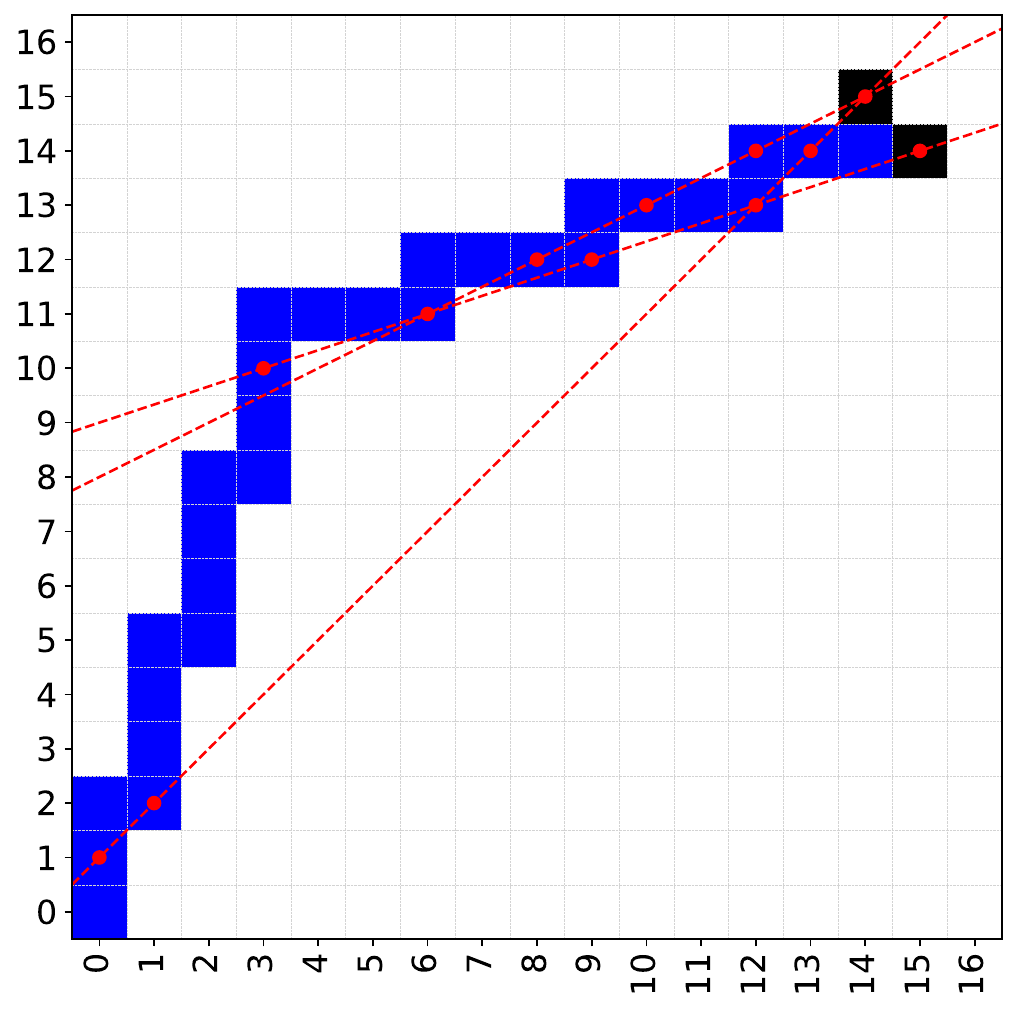}
    \caption{The unique longest GR(5) walk up to isomorphism;
    any additional step (black) creates five collinear points.}
    \label{fig:28steps}
\end{figure}

The problem of finding maximal GR($k$) walks is difficult because the search space grows exponentially in the length of the walk:
each $m$-step walk chooses north or east at each step, yielding~$2^m$ possibilities. 
Establishing $a(k)=m$ requires showing both that every $m$-step walk contains $k$ collinear points (thereby showing $a(k)\leq m$) and
that some $(m-1)$-step GR($k$) walk exists (thereby showing $a(k)\geq m$).
Without loss of generality, to prove $a(k)\leq m$ only walks starting at the origin $(0,0)$ and ending on the line $y=m-x$ need to be considered.

In this paper, we improve the lower bound on $a(7)$ by finding a GR(7) walk with 327 steps.
We also enumerate all GR($k$) walks for $k\leq6$, and as a consequence confirm Shallit's computation of $a(k)$ for $k\leq6$.
Our searches and enumerations for GR($k$) walks are performed using a satisfiability (SAT) solver.
SAT solvers are exceptional general-purpose search tools~\cite{Bright2022}
and have been surprisingly effective at solving problems in discrete geometry~\cite{Heule2024,Subercaseaux2025,Subercaseaux2024},
finite geometry~\cite{LamPaper}, infinite graph theory~\cite{Subercaseaux2023}, and various puzzles in discrete mathematics~\cite{Bright2020}.
In fact, certain combinatorial problems with enormous search spaces (like the Boolean Pythagorean triples
problem~\cite{Heule2017}) have \emph{only} been solved using a SAT solver, despite at first glance having nothing to do with Boolean logic.
SAT solvers also have several other advantages over searching with custom-written code:
from a correctness perspective, they provide \emph{proof certificates} when the object being searched for does not exist.
For example, when the SAT solver determines that there exists no $m$-step GR($k$) walk, it also
provides a certificate that can be certified by an independent proof verifier, a relatively simple
piece of software.  Consequently, only the proof verifier needs to be trusted, not the SAT solver itself.
The SAT encoding also needs to be trusted, but it tends to be simpler to write a SAT encoding than it is to write optimized search code.
We describe our SAT encoding for the Gerver--Ramsey collinearity problem in Section~\ref{sec:satisfiability}.

The primary contribution of this paper is a SAT-based method of finding long GR($k$) walks and an
experimental study of the Gerver--Ramsey problem for $k\leq7$.  In particular, we
enumerate all GR($k$) walks up to isomorphism for $k\leq6$.
In the process, we find all maximal GR($k$) walks for $k\leq6$, accompanied by proof certificates that no longer GR($k$) walks exist (see Section~\ref{sec:enumeration-kleq6}).
We also determine the north-most and east-most GR(7)-reachable points using up to 267 steps, as well as the first GR(7)-unreachable
point on the line $y=x+1$ (see Section~\ref{sec:gr7reachability}).
Lastly, we find a GR(7) walk of length 327 (see Section~\ref{sec:gr7improvement}), improving on the previously longest known GR(7) walk of length 260~\cite{A231255}.

\section{Satisfiability Solving}\label{sec:satisfiability}

The Boolean satisfiability problem (SAT) asks whether a Boolean formula admits an assignment of truth values to its variables making the formula evaluate to true.
It was the first problem proven to be NP-complete~\cite{Cook1971}, and it remains a cornerstone of computational complexity theory.
Over the last several decades, an active research community has developed increasingly efficient SAT solvers.
Modern SAT solvers require the input formula to be specified in a format known as conjunctive normal form (CNF) defined in terms
of literals and clauses.
A \emph{literal} is a Boolean variable $p$ or its negation $\lnot p$.
A \emph{clause} is a disjunction of literals.
A formula is in \emph{CNF} if it is a conjunction of clauses.
For example, $(p \lor q) \land (p) \land (\lnot p \lor q)$ is in CNF and contains three clauses, the second of which is a \emph{unit clause} (consisting of a single literal).
We may also use the implication connective to express clauses with the understanding that $p\limp q$ is shorthand for $\lnot p\lor q$.
A \emph{satisfying assignment} of a formula is a true/false assignment to the variables of the formula
such that the whole formula evaluates to true under that truth assignment (i.e.,
every clause contains at least one literal assigned to true).

To apply a SAT solver to a search problem, the problem must be encoded as a CNF formula in such a way that
the search problem has a solution if and only if the CNF formula has a satisfying assignment
(i.e., the formula is \emph{satisfiable}).
Moreover, it should be possible to take a satisfying assignment of the formula and
translate it into a solution of the search problem.
Conversely, if the SAT solver determines the CNF formula has no satisfying assignment
(i.e., the formula is \emph{unsatisfiable})
this implies the search problem has no solution.

We describe our encoding of the Gerver--Ramsey collinearity problem into conjunctive normal form in Section~\ref{sec:encoding},
and describe an encoding of the problem into a more general format called ``at-least-$k$ conjunctive normal form'' in Section~\ref{sec:knf}.
We also found that using some clauses encoding the reachability of points was useful---although not strictly
required, they improved the solving times in practice (see Section~\ref{sec:reachability}).
Additionally, we found some of the constraints in our encoding were redundant in practice and removing them
made the solver more efficient (see Section~\ref{sec:heuristic}).
Finally, we describe our process of parallelization using a technique known as cube-and-conquer in Section~\ref{sec:c&c}.
Cube-and-conquer was necessary in order to solve the hardest SAT instances in a reasonable amount of time.

\subsection{SAT encoding}\label{sec:encoding}

We now describe our SAT encoding asserting the existence of a GR($k$) walk with $m$ steps.
The values of $k$ and $m$ are taken to be fixed in advance, and
we let $n=m+1$ denote the number of points in the walk.
Without loss of generality we take our starting point as the origin $(0,0)$, so
an $m$-step north--east lattice path ends on the line $y=m-x$.
In order to describe such a walk, we use the Boolean variable $v_{x,y}$ to represent
that the point $(x,y)$ appears on the walk.  There are $i+1$ points
reachable from the origin in a north--east lattice path with $i$ steps,
so there are a total of $\sum_{i=0}^{n-1}(i+1)=n(n+1)/2$
point variables in our instance.

\subsubsection{Path constraints}

Next, we describe the constraints on the variables $v_{x,y}$ that must hold
in north--east lattice paths.  In what follows, $(x,y)$ is one of the
points that might appear on an $m$-step ($n$-point) north--east lattice path (i.e., $(x,y)\in\Enn^2$
with $x+y\leq m$).
First, we know that if point $(x,y)$ is on the path
then either $(x+1,y)$ or $(x,y+1)$ is on the path, unless $(x,y)$ is the final point.
We encode this constraint using
the clauses
\[ v_{x,y} \limp (v_{x+1,y} \lor v_{x,y+1}) \qquad\text{for $x+y\neq m$.} \]
Second, we know that if point $(x,y)\neq(0,0)$ is on the path
then either $(x-1,y)$ or $(x,y-1)$ is on the path.  We encode this using
the clauses
\[ v_{x,y} \limp (v_{x-1,y} \lor v_{x,y-1}), \quad v_{0,y} \limp v_{0,y-1}, \quad v_{x,0}\limp v_{x-1,0} \qquad\text{for $x,y\geq1$.} \]
Third, we know that a path never splits into two directions:
both $(x+1,y)$ and $(x,y+1)$ can never both be on the path at the same time.
We encode this using the clauses
\[ \lnot v_{x+1,y}\lor\lnot v_{x,y+1} \qquad\text{for $x+y\neq m$.} \]
Lastly, we assert that the origin is on the path with the unit clause $v_{0,0}$.
In total, we have $O(n^2)$ path constraints.
A satisfying assignment of these constraints provides a north--east lattice path
starting from the origin and ending after $m$ steps.

\subsubsection{Non-collinearity constraints}
\label{sec:non-collinearity}

In order to assert that the path is a GR($k$) walk, we need to assert
that it does not contain~$k$ collinear points.  To do this, we use a
generalization of clauses
known as cardinality constraints.  An at-most-$k$ cardinality constraint
over a set of literals $X=\{x_1, \dotsc, x_s\}$
says that no more than~$k$ of the literals in $X$ can be assigned true, and
we use the notation $\sum_{i=1}^s x_i \leq k$ to represent this constraint.
Note that cardinality constraints are not natively in conjunctive normal form. However,
there are a number of efficient ways of converting cardinality
constraints into CNF such as the sequential counter encoding~\cite{Sinz2005} and the totalizer encoding~\cite{totalizer}.
Moreover, other formats like at-least-$k$ conjunctive normal form~\cite{Reeves2025}
have native support for cardinality constraints (see Section~\ref{sec:knf}).

Now, we assert that no $k$ collinear points appear on the path.  This requires determining
all ways in which the points in our instance (i.e., $(x,y)\in\{0,\dotsc,n-1\}^2$ with $x+y<n$)
might lie on the same line.  First, consider the case of avoiding $k$ points on the same
vertical line.  Avoiding $k$ points on the line $x=i$ can be accomplished with the cardinality constraint
\[ \sum_{j=0}^{n-i-1} v_{i,j} \leq k - 1 , \]
and avoiding $k$ points on the horizontal line $y=j$ can be accomplished with the
cardinality constraint
\[ \sum_{i=0}^{n-j-1} v_{i,j} \leq k - 1 . \]
Generalizing this, avoiding $k$ points on the line with slope $s$ and $y$-intercept $b$
can be accomplished with the cardinality constraint
\[ \sum_{\substack{i,si+b\in\Enn\\i<(n-b)/(s+1)}} v_{i,si+b} \leq k - 1 . \]
The slope $s$ will be of the form $r/d$ where $r\in\Enn$ is the rise of the slope
and $d\in\Enn$ is the run of the slope.
We can assume that $r+d\leq (n-1)/(k-1)$, since otherwise there will necessarily be fewer than
$k$ points on a line with slope $s=r/d$ in the relevant region, as indicated in
Proposition~\ref{prop:slope2}.
\begin{proposition}\label{prop:slope2}
Suppose the region $\{\,(x,y)\in\Enn^2:x+y<n\,\}$ contains $k$ points on a line with
slope $r/d$.  Then $r+d\leq (n-1)/(k-1)$.
\end{proposition}
\begin{proof}
Order the $k$ points on a line with slope $r/d$ so that $(x_{i+1},y_{i+1})=(x_1+id,y_1+ir)$
for $0\leq i<k$.
The number of steps it takes to walk from $(x_1,y_1)$ to $(x_k,y_k)$ is
$x_k-x_1+y_k-y_1=d(k-1)+r(k-1)=(r+d)(k-1)$.  It is possible to take up to $n-1$ north--east steps while
remaining in the region $\{\,(x,y)\in\Enn^2:x+y<n\,\}$, so
if all points are in this region then $(r+d)(k-1)\leq n-1$, and $r+d\leq(n-1)/(k-1)$.
\end{proof} 
We can also assume that the slope
is written in lowest terms so that $r$ and~$d$ are coprime.  The probability
two integers in $[1,n)$ are coprime tends to $6/\pi^2$ when $n\to\infty$,
so there are asymptotically
$O((n/k)^2)$ slopes to consider when the rise and run are each bounded
by $n/k$ (cf.~Proposition~\ref{prop:slope3} below).

Each slope $s=r/d$ forms a line $y=sx$ for which we add the constraint $\sum_{i=0}^{\lceil n/(s+1)\rceil-1}v_{i,si}\leq k-1$.
There are $O(n^2/k^2)$ slopes to consider, so there are $O(n^2/k^2)$ lines to consider with zero $y$-intercept.
Next, consider lines with positive $y$-intercepts $b\in\Que$.
Note $b\leq n$, otherwise $y=si+b>n$.  Also,
the denominator of $b$ in lowest terms must divide $d$,
otherwise $y=(r/d)i+b$ would not be an integer.  Thus we have $b=\alpha/d$ where
$\alpha\leq nd$ is a positive integer.
Thus, there are $O(nd)=O(n^2/k)$ lines with positive $y$-intercepts to consider.  By symmetry,
there are also $O(n^2/k)$ lines with positive $x$-intercepts to consider (which are equivalent to the lines
with $b<0$).  Thus, in total there are $O(n^2/k)$ values of $b$
to consider.

Since there are $O(n^2/k^2)$ slopes $s$ to consider and $O(n^2/k)$ $y$-intercepts $b$ to consider,
there are $O(n^4/k^3)$ cardinality constraints in total.
Not all these constraints are necessary to include; e.g., if there are
strictly less than $k$ variables $v_{x,y}$ corresponding to points on the line $y=sx+b$,
then the constraint associated with this line is unnecessary, since
the constraint $\sum_{y=sx+b}v_{x,y}<k$ will always be satisfied.
Proposition~\ref{prop:maxslope} below provides an additional
restriction on the line slopes, showing certain
non-collinearity constraints to be unnecessary.

\begin{proposition}\label{prop:maxslope}
A north--east lattice path without $k-1$ consecutive
steps in the same direction does not have $k$ points on a line with
slope $s>k-2$.
\end{proposition}

\begin{proof}
For contradiction, suppose a north--east lattice path without $k-1$ consecutive steps in the same
direction has $k$ points $(x_1,y_1)$, $\dotsc$, $(x_k,y_k)$ on a line with slope $s>k-2$.
Suppose $x_k-x_1=r$ and
let $V_i$ denote the number of north steps taken along the line $x=i$, noting that
$V_i\leq k-2$ since the path never has $k-1$ consecutive north steps.
The total number of north steps
taken from $(x_1,y_1)$ to $(x_k,y_k)$ is $y_k-y_1=\sum_{i=x_1}^{x_k}V_i$.
Since $V_i\leq k-2$ and there are $r+1$ summands,
$\sum_{i=x_1}^{x_k}V_i\leq (r+1)(k-2)$.

Also note that $y_{i+1}-y_i=s(x_{i+1}-x_i)$ since $(x_i,y_i)$ and $(x_{i+1},y_{i+1})$
are both on a line with slope $s$.  Since $s>k-2$, we have $y_{i+1}-y_i>(k-2)(x_{i+1}-x_i)$,
and since both sides are integers, we have $y_{i+1}-y_i\geq(k-2)(x_{i+1}-x_i)+1$.
Summing this from $i=1$ to $k-1$ we obtain
\begin{align*}
\sum_{i=1}^{k-1}(y_{i+1}-y_i) &\geq \sum_{i=1}^{k-1}\bigl((k-2)(x_{i+1}-x_i)+1\bigr) \\
&= (k-2)(x_k-x_1)+(k-1) \\
&= r(k-2)+(k-1) = (r+1)(k-2)+1 .
\end{align*}
However, the left-hand side equals $y_k-y_1=\sum_{i=x_1}^{x_k}V_i$.  Putting our bounds on $y_k-y_1$ together,
\[ (r+1)(k-2) + 1 \leq y_k-y_1 \leq (r+1)(k-2) ,  \]
a contradiction.
\end{proof}

As a result of Proposition~\ref{prop:maxslope}, we do not need to consider cardinality
constraints corresponding to lines with slopes $s>k-2$.  Symmetrically,
we also do not need to consider constraints corresponding to lines with slopes $s<1/(k-2)$.
Even when the encoding is optimized to remove constraints proven to be
unnecessary, the cardinality constraints still dominate the encoding size. For example,
with $n=325$ and $k=7$ there are about $953{,}000$ non-collinearity constraints
and about $158{,}000$ path constraints.

Finally, Proposition~\ref{prop:maxslope} allows the bounds on the rise and run in
Proposition~\ref{prop:slope2} to be slightly tightened.
\begin{proposition}\label{prop:slope3}
Suppose the region $\{\,(x,y)\in\Enn^2:x+y<n\,\}$ contains $k>2$ points on a line with
slope $r/d\leq k-2$.  Then $r < (n-1)/k$.
\end{proposition}
\begin{proof}
By Proposition~\ref{prop:maxslope}, we have $r/(k-2)\leq d$, and by Proposition~\ref{prop:slope2}, we have
$r+d\leq(n-1)/(k-1)$.  Combining these, we have $r(1+1/(k-2))\leq(n-1)/(k-1)$.  Multiplying both sides
by $(k-2)/(k-1)$ gives $r\leq(n-1)(k-2)/(k-1)^2<(n-1)/k$.
\end{proof}
Similarly, it also follows that $d<(n-1)/k$.  Thus we need only consider
lines whose slopes have rise and run in $\{1,\dotsc,\lfloor (n-1)/k\rfloor\}$.

\paragraph{Horizontal/vertical non-collinearity optimization}

The only way it is possible to have~$k$ vertical collinear points in a north--east lattice path
is by taking a north step $k-1$ times in a row.  Thus, it is possible to block the existence
of $k$ vertical points on the line $x=i$ via the binary clauses
\[ v_{i,j} \limp \lnot v_{i,j+k-1} \qquad\text{for all $j=0,\dotsc,n-i-k$.} \]
Similarly, it is possible to block the existence of $k$ horizontal points on the line $y=j$ via
the binary clauses
\[ v_{i,j} \limp \lnot v_{i+k-1,j} \qquad\text{for all $i=0,\dotsc,n-j-k$.} \]
Although a minor optimization, experimentation showed that this alternate encoding
of the vertical and horizontal non-collinearity constraints tended to be preferable
to the cardinality encodings described above.

\subsubsection{Symmetry breaking}\label{sec:symbreak}

There are three nontrivial operations that when applied to a GR($k$) walk produce another GR($k$)
walk: complement the steps (switch north steps with east steps and vice versa),
reverse the steps, and a complement + reverse combination.
We consider the paths generated by
these operations as equivalent, and to shrink the search space it is desirable to remove such
paths from the search space---assuming that \emph{up to equivalence} we don't remove
paths.  The process of adding extra constraints
that remove solutions that can be assumed without loss of generality is known
as \emph{symmetry breaking}.

The complementation symmetry is simple to break by enforcing the first step to be north, since
a Gerver--Ramsey walk without a north first step can be complemented to
make an equivalent walk with a north first step.
We encode this by adding the unit clause $v_{0,1}$ into our SAT encoding (which then
implies $\lnot v_{1,0}$).

We also tried breaking the reversal and reversal+complement symmetries, but the SAT encoding
to do this was more involved.  Ultimately, experiments revealed that the overhead
of adding more clauses into the encoding was more trouble than it was worth.  Thus, the
SAT encodings we used ignored the reversal symmetries, only
breaking the complementation symmetry via a north first step.
Exhaustive enumeration of GR($k$) walks was possible for $k\leq6$,
and after the enumeration was complete we checked for and removed GR($k$) walks
that were duplicates up to equivalence (see Section~\ref{sec:enumeration-kleq6}).

\subsubsection{Blocking extremal points}\label{sec:extremalpoints}

Points that are too close to the $x$-axis or $y$-axis can quickly be shown to never
occur in a GR($k$) walk and therefore can be blocked directly in the encoding.
In particular, because a GR($k$) walk can never take $k-1$ consecutive steps in the north direction,
a GR($k$) walk can never cross the line $y=(k-2)x+(k-1)$.  Thus, we add in the unit clauses
\[ \lnot v_{x,(k-2)x+k-1} \qquad\text{for $0\leq x< n/(k-1)-1$} \]
which blocks the points $(0,k-1)$, $(1,2k-3)$, $(2,3k-5)$, $\dotsc$\ from appearing on the path.

Similarly, a GR($k$) walk can never take $k-1$ consecutive steps in the east direction,
so a GR($k$) walk whose first step is north can never cross the line
$y=\frac{1}{k-2}(x-1)$.  Thus, we add in the unit clauses
\[ \lnot v_{(k-2)y+1,y} \qquad\text{for $0\leq y<(n-1)/(k-1)$} \]
which blocks the points $(1,0)$, $(k-1,1)$, $(2k-3,2)$, $\dotsc$\ from appearing on the path.

\subsection{At-least-$k$ conjunctive normal form}
\label{sec:knf}

It is not always convenient to write a logical expression in CNF, and a number of more expressive extensions of CNF
have been proposed.  One such extension proposed by Reeves, Heule, and Bryant~\cite{CardinalityCadical},
called \emph{at-least-$k$ conjunctive normal form} (KNF), augments CNF with constraints of the form
$l_1+\dotsb+l_s\geq k$ where $l_1$, $\dotsc$, $l_s$ are literals.  Such a constraint is known as a \emph{klause}
and is satisfied by an assignment if at least $k$ of the literals $l_1$, $\dotsc$, $l_s$ are true
under that assignment.

Klauses are by definition lower bounds, but upper bounds can also be expressed as klauses
since the lower bound $l_1+\dotsb+l_s\geq k$ is equivalent to the upper bound
$\bar l_1+\dotsb+\bar l_s\leq s-k$ where $\bar x$ denotes the negation of $x$.
Thus, we are able to express the non-collinearity constraints from Section~\ref{sec:non-collinearity}
natively using klauses.  For example, the constraint that there are at most $k-1$ points
on the line $y=x$ is represented as the klause
\[ \sum_{i=0}^{\lceil n/2 \rceil - 1} \bar{v}_{i,i} \geq \lceil n/2 \rceil - k + 1 . \]

Reeves et al.~\cite{CardinalityCadical} provide a KNF solver based on the SAT solver \CaDiCaL~\cite{CadicalPaper}
called Cardinality-{\CaDiCaL}\@.\footnote{Code available at \url{https://github.com/jreeves3/Cardinality-CDCL/}.}
Their solver is able to reason about klauses natively and incorporates a cardinality-based
propagation routine for deriving consequences of partial assignments.
They report that Cardinality-{\CaDiCaL} performs particularly well on
satisfiable instances having many large cardinality constraints.
Intuitively, cardinality-based propagation allows solving satisfiable instances faster because
it bypasses the auxiliary variables used to convert cardinality constraints into CNF in a compact way.
On the other hand, they report that for unsatisfiable instances converting klauses
into CNF cardinality constraint encodings tends to result in improved performance.
Intuitively, the auxiliary variables used in the CNF conversion
tend to be important for finding short proofs of unsatisfiability.

\subsection{Encoding the unreachability of points}
\label{sec:reachability}

Section~\ref{sec:extremalpoints} provides upper and lower bounds on GR($k$)-reachable points; in particular,
GR($k$)-reachable points always lie between the lines $y=(k-2)x+(k-1)$ and $y=\frac{1}{k-2}(x-1)$.  However,
these bounds are not tight for large $x$.  Given that if a point $(x,y)$
can be shown to be GR($k$)-unreachable then the Boolean variable $v_{x,y}$ can be fixed to false, it is
desirable to determine the reachability of as many points as possible.

The problem of determining the reachability of a point can be phrased using a variant of the
SAT encoding we've already described.  Say we want to determine if $(x,y)$ is a GR($k$)-reachable point.
We generate the SAT encoding specifying the existence of an $(x+y)$-step GR($k$) walk, except we add the additional
unit clause $v_{x,y}$ into the encoding.  The presence of $v_{x,y}$ ensures the point $(x,y)$ must appear
on the path.  If such an instance is satisfiable, the satisfying assignment will provide a GR($k$) walk from
$(0,0)$ to $(x,y)$.  Otherwise, if such an instance is unsatisfiable, this implies that $(x,y)$ is a
GR($k$)-unreachable point.

Once it is known that $(x,y)$ is GR($k$)-unreachable, the unit clause $\lnot v_{x,y}$ is included in
all larger instances specifying the existence of GR($k$) walks with more than $x+y$ steps.  Such unit clauses
help the solver, because the solver now no longer needs to consider walks passing through $(x,y)$.
In fact, we can say more: if $(x,y)$ is GR($k$)-unreachable when starting from the origin, it must also be the
case that $(x+x_0,y+y_0)$ is GR($k$)-unreachable when starting from $(x_0,y_0)$.  Thus, instances
asserting the existence of GR($k$) walks of length $n$ may also include clauses of the form
\begin{equation}
v_{x_0,y_0}\limp\lnot v_{x_0+x,y_0+y} \qquad\text{for all $(x_0,y_0)\in\Enn^2$ with $x_0+y_0<n-x-y$} \label{eq:binreachable}
\end{equation}
where $(x,y)$ is a GR($k$)-unreachable point with $x+y<n-1$.

Some care must be taken in order to use the unreachability clauses in conjunction with the symmetry breaking
described in Section~\ref{sec:symbreak}.
Recall our symmetry breaking assumes the first step is north.  If $(x,y)$ is determined to be
unreachable in a GR($k$) walk using $x+y$ steps (the first of which is north),
it follows that $(x,y)$ does not appear on all longer GR($k$) walks using our symmetry breaking (i.e., with a north first step).
However, the clauses in~\eqref{eq:binreachable} can only be added if it is known
that $(x,y)$ is unreachable from the origin in walks with \emph{either} a north or east first step.
Note that walks from the origin to $(x,y)$ with a north
first step are equivalent to walks from the origin to $(y,x)$ with an east first step.
Thus, we add clauses of the form~\eqref{eq:binreachable} when both $(x,y)$ and $(y,x)$
were determined to be unreachable from the origin in GR($k$) walks with a north first step.

\subsection{Constraint-removal heuristic}\label{sec:heuristic}

Although not always the case,
SAT solvers may perform better if redundant constraints are not used in
the encoding.  During solving, modern SAT solvers store all clauses in memory
and most of their time is spent performing constraint propagation, a task whose
running time is proportional to the number of clauses stored in memory.  Thus, reducing
the number of stored clauses often improves the performance of the solver, particularly
when the removed clauses are redundant.

In the Gerver--Ramsey collinearity problem, we observed that many of the
non-collinearity constraints described in Section~\ref{sec:non-collinearity} were
not useful in practice.  That is, many of these constraints could be removed and
the solver could still either find correct GR($k$) paths (in satisfiable instances) or
prove that no GR($k$) paths exist (in unsatisfiable instances).

Note that the technique of removing some non-collinearity constraints is
heuristic.  On the one hand, if constraints are removed from the instance
and the solver reports an UNSAT result, we know for certain that the original instance
was also unsatisfiable (as removing constraints can only \emph{increase} the number of satisfying
assignments, never decrease it).  On the other hand, if constraints are removed from
the instance and the solver reports a SAT result,
there is no guarantee that the satisfying assignment returned by the solver
is actually a GR($k$) path.  However, in such cases we can explicitly
check that the returned path has no~$k$ collinear points on it.  To do this check
efficiently, we iterate over all pairs of points on the path and ensure that the
line through the two points in the pair contains fewer than $k$ points on the path.

In practice, we found that the majority of the non-collinearity constraints
corresponded to lines having a relatively small number of points $(x,y)$
in the relevant region of $\{\,(x,y)\in\Enn^2:x+y<n\,\}$.  For example, for $k=7$ and $n=300$,
about 35\% of the constraints contain exactly 7 points in the relevant region. In practice
these constraints are unlikely to be useful, as
it is unlikely for a satisfying assignment to pass through all
7 points simultaneously.  Thus, we removed constraints corresponding
to lines with a small number of points in the relevant region.
For $k=7$ and $n\geq150$, we removed all constraints corresponding
to lines containing 16 or fewer points in the relevant region. Despite removing the majority of the non-collinearity constraints in this way (for example, about 94\% for $n=300$), 
most satisfying assignments found by the solver produced a valid GR($k$) walk, with fourteen exceptions among several hundred cases.
This heuristic also produced a
speedup in the solver's efficiency (see Section~\ref{sec:benchmarking-heuristic}).

\subsection{Parallelization}
\label{sec:c&c}

As the number of steps in the path increases, the SAT instances
tend to become more difficult.  The largest instances we solve
are so difficult that solving them using a single processor would
take an infeasible amount of time.  Thus, in order to make progress it is necessary
to exploit parallelization and have multiple processors working
on solving the SAT instance in parallel.

One of the most successful parallelization techniques in SAT solving
is known as cube-and-conquer~\cite{Heule2012CubeAndConquer}.
This technique aims to divide the search space into disjoint
subproblems of roughly balanced difficulty.  If this can be achieved, multiple processors can
solve subproblems independently, providing a speedup
proportional to the number of processors available.  The method uses what is
known as a lookahead solver to determine how to split the SAT instance
by ``branching'' on a variable in the instance---setting the variable
to true in one subproblem and false in another
subproblem.  A lookahead solver spends a significant amount of time
determining which variable is best to split on in order to split the
problem into two subproblems of roughly equal difficulty.

We use Heule's lookahead solver {\march} in our work~\cite{Heule2005}.
After a variable is selected to branch on, {\march} generates a subproblem
in which the variable is true and a subproblem in which the variable is false,
and then applies Boolean constraint propagation (i.e., derives consequences
of fixing the variable in each subproblem).  The process then repeats
recursively until a set number of cubes have been created or the subproblems have been determined
to be so easy to solve that splitting them further is no longer necessary.

A \emph{cube} is a conjunction of literals, e.g., $x_1\land \lnot x_2\land x_3$.
The above process of splitting can be viewed as generating a collection of cubes
that partition the search space, with each cube defining a single subproblem.
Each processor is provided the original SAT instance along with one or more cubes to solve.
For each cube, the
SAT solver assumes the literals in the cube are each true by adding each literal
in the cube as a unit clause.

\section{Results}\label{sec:results}

We now discuss our experimental results.\footnote{Our code is available at \url{https://github.com/aaronbarnoff/Collinear}. 
}
We start with a description of the benchmarking
we did in order to determine the effectiveness of
the encodings from Section~\ref{sec:satisfiability} (see
Section~\ref{sec:benchmarking}).
We then describe how we used our SAT encoding to enumerate
all GR($k$) walks up to $k=6$ and produce certificates
demonstrating that there do not exist GR(3), GR(4), GR(5), and GR(6) walks
with 4, 9, 29, and 97 steps, respectively (see Section~\ref{sec:enumeration-kleq6}).
Unless otherwise mentioned, experiments in Sections~\ref{sec:benchmarking},~\ref{sec:enumeration-kleq6}, and~\ref{sec:gr7reachability} were run on the Digital Research Alliance of Canada Fir cluster,
a high performance computing cluster consisting of
AMD EPYC processors, most of which run at 2.7~GHz.
The experiments of Section~\ref{sec:gr7improvement} used the Digital Research Alliance of Canada Nibi cluster with Intel Xeon processors at 2.4~GHz.
A few experiments were run on a desktop computer with an AMD Ryzen 9950X 4.3~GHz processor and 64~GiB of memory
and these are indicated separately.

\subsection{Benchmarking}\label{sec:benchmarking}

Because modern SAT solvers use a number of heuristics
(e.g., to decide which variable to branch on when solving) their solving times on the same instance
tend to vary widely.  This is especially true if the instance is satisfiable, since the solver
will stop as soon as a single satisfying assignment is found, and depending on the heuristic
choices this will sometimes happen much quicker than usual.  Unsatisfiable instances usually have
more consistent solve times (since in these cases the solver always needs to prove there are no satisfying assignments)
but even unsatisfiable instances have variance in their solve times.
To mitigate the effect of this variance, we solved each benchmark 15 times
using 15 different random seeds and report the median running time.

\subsubsection{KNF vs.\ CNF: Performance}\label{sec:knfcnfperformance}

Our first set of benchmarks explore the performance of a KNF encoding versus
a CNF encoding.  We experimented with all of the CNF cardinality encodings
supported by the Python library PySAT~\cite{PySAT}, and we found the sequential counter
encoding had the best performance for this problem, so our CNF instances used the
sequential counter encoding for the non-collinearity cardinality constraints.
The SAT solver used to solve the CNF instances was {\CaDiCaL}~\cite{CadicalPaper}, and the KNF solver
was Cardinality-{\CaDiCaL}~\cite{CardinalityCadical}.

Four satisfiable benchmarks were chosen
and four unsatisfiable benchmarks were chosen (in each case, one benchmark
had $k=6$, and the other three benchmarks had $k=7$).  The unsatisfiable instances with $k=7$
had an endpoint $(x,y)$ of the path added as a unit clause $v_{x,y}$, with the point
$(x,y)$ chosen to be GR($k$)-unreachable.
Each benchmark was solved 15 times using 15 random seeds, and the median times (in seconds)
are presented in Table~\ref{tab:cnfknf}.
{\CaDiCaL} used at most 7743~MiB of memory on the satisfiable benchmarks
and 1685~MiB on the unsatisfiable benchmarks.
For the KNF encoding, Cardinality-{\CaDiCaL} required at most 887~MiB on
the satisfiable benchmarks and 492~MiB on the unsatisfiable benchmarks.

\begin{table}
\caption{Median solve times (in seconds) across 15 trials for two encodings (a CNF encoding and a KNF encoding)
of eight different benchmarks.}
\centering
\begin{tabular}{c c c c S[table-format=5.1] S[table-format=5.1]}
$k$ & $n$ & Endpoint  & Type  & {CNF time} & {KNF time} \\ \hline
6   & 97  & ---       & SAT   & 102.5      & 25.5       \\
7   & 220 & ---       & SAT   & 2347.4     & 76.2       \\
7   & 240 & ---       & SAT   & 6802.7     & 374.6      \\
7   & 261 & ---       & SAT   & 26499.8    & 2750.0     \\ \hline
6   & 98  & ---       & UNSAT & 616.6      & 361.0      \\
7   & 122 & $(33,88)$ & UNSAT & 363.8      & 702.2      \\
7   & 151 & $(46,104)$& UNSAT & 1529.3     & 10826.6    \\
7   & 180 & $(56,123)$& UNSAT & 2842.6     & 46439.4    \\
\end{tabular}
\label{tab:cnfknf}
\end{table}

The results show that the KNF encoding tends to perform better on satisfiable instances, while the CNF
encoding tends to perform better on unsatisfiable instances, matching the observation of
Reeves et al.~\cite{CardinalityCadical}.  Thus, in our future results when we know or expect
the instance to be satisfiable we use the KNF encoding and otherwise we use the CNF encoding.

\subsubsection{Unreachable points encoding: Performance}\label{sec:reachability-performance}

The next set of benchmarks examines the performance of the unreachable point encoding
described in Section~\ref{sec:reachability}.  When solving an instance
asserting the existence of an $m$-step GR($k$) walk, if there are known
points $(x,y)$ with $x+y<m$ that are GR($k$)-unreachable, we exploit
this in our encoding.  For now, we assume the reachability of points in $\{\,(x,y)\in\Enn^2:x+y<n-1\,\}$
is known; in Section~\ref{sec:gr7reachability} we describe the
computations performed to determine the reachability of points.
As described in Section~\ref{sec:reachability}, for each
GR($k$)-unreachable point $(x,y)$,
we add the unit clause $\lnot v_{x,y}$ (stating that $(x,y)$ is not on the path)
and the binary clauses from~\eqref{eq:binreachable}.

We used the same eight benchmarks from Section~\ref{sec:knfcnfperformance}, and solved
them with and without the unreachability clauses.  Again, 15 trials were run
with 15 different random seeds and the median running time is given in Table~\ref{tbl:unreachability}.
{\CaDiCaL} used at most 970~MiB on the unsatisfiable benchmarks,
whereas Cardinality-{\CaDiCaL} used at most 597~MiB on the satisfiable benchmarks.

The results show that the unreachability clauses tend to help the solver, especially for
unsatisfiable instances, where adding the unreachability clauses improved the solver's
median running time, sometimes dramatically: for example, the solver was able to show there is no GR($7$) walk
from the origin to $(56,123)$ about 110 times faster when the unreachable point clauses were included.
For satisfiable instances, the results were less dramatic, but the unreachability clauses
still tended to improve the performance of the solver.

\begin{table}
\caption{A table summarizing solve times using our encoding
with and without the reachability clauses.
The reported times are the median time (in seconds) across 15 trials, each
using a different random seed.
Satisfiable (SAT) instances use the KNF encoding and unsatisfiable (UNSAT) instances
use the CNF encoding.}
\centering
\begin{tabular}{cccc S[table-format=4.1] S[table-format=4.1]}
$k$ & $n$ & Endpoint  & Type  & Without & With \\ \hline
6 & 97  & ---         & SAT   & 25.5   & 22.9 \\
7 & 220 & ---         & SAT   & 76.2   & 59.3 \\
7 & 240 & ---         & SAT   & 374.6  & 361.3 \\
7 & 261 & ---         & SAT   & 2750.0 & 2397.4 \\ \hline
6 & 98  & ---         & UNSAT & 616.6  & 508.7 \\
7 & 122 & $(33,88)$   & UNSAT & 363.8  & 134.4 \\
7 & 151 & $(46,104)$  & UNSAT & 1529.3  & 35.2  \\
7 & 180 & $(56,123)$  & UNSAT & 2842.6 & 25.8  \\
\end{tabular}
\label{tbl:unreachability}
\end{table}

\subsubsection{Constraint-removal heuristic: Performance}\label{sec:benchmarking-heuristic}

We now examine the performance of the constraint-removal heuristic
described in Section~\ref{sec:heuristic}.  For $k=6$ with the heuristic enabled, we ignored
all non-collinearity constraints corresponding to lines with 13 or fewer
points in the region $\{\,(x,y)\in\Enn^2:x+y<n\,\}$ and the solver was still able to prove
there are no GR($k$) paths with $n=98$ points.  For $n=97$, the solver found paths containing
$k$ collinear points, indicating that some of the ignored constraints were not redundant.
Since the $k=6$ instances were quickly solvable without the heuristic anyway,
for $k=6$ we only enabled the heuristic on the final unsatisfiable case (for $n=98$ points).

For $k=7$, we ignored all non-collinearity constraints corresponding to lines with 16 or fewer
points in the region $\{\,(x,y)\in\Enn^2:x+y<n\,\}$.  For $n\geq150$ the solver was able to
solve our previous benchmarks correctly---for the satisfiable instances, valid GR($k$) paths
were found and the unsatisfiable instances were determined to have no GR($k$) paths.
For the unsatisfiable $n=122$ benchmark, the solver found a satisfying assignment
having $k$ collinear points on the corresponding path, meaning the ignored constraints
were not redundant.  This is an indication that the heuristic works better for larger $n$
which are the instances that are the most difficult to solve.

Table~\ref{tbl:heuristic-results}
contains a summary of the results we found using our constraint-removal heuristic on
several satisfiable and unsatisfiable benchmarks. The constraint-removal heuristic significantly lowered memory usage, bringing {\CaDiCaL} down to at most 391~MiB on the unsatisfiable cases (2.5$\times$ reduction) and Cardinality-{\CaDiCaL} to at most 344~MiB on the satisfiable ones (1.7$\times$ reduction).

\begin{table}
\caption{A table summarizing solve times with
and without using our constraint-removal heuristic.
The reported times are the median time (in seconds) across 15 trials, each
using a different random seed.
Satisfiable (SAT) instances use the KNF encoding and unsatisfiable (UNSAT) instances
use the CNF encoding.
}
\centering
\begin{tabular}{cccc S[table-format=4.1] S[table-format=4.1]}
$k$ & $n$ & Endpoint  & Type  & {Heuristic Off} & {Heuristic On} \\ \hline
7 & 220 & ---         & SAT   & 59.3   & 48.4 \\%
7 & 240 & ---         & SAT   & 361.3  & 245.2 \\
7 & 261 & ---         & SAT   & 2397.4 & 1710.3 \\ \hline
6 & 98  & ---         & UNSAT & 508.7  & 395.9  \\
7 & 151 & $(46,104)$  & UNSAT & 35.2   & 26.4   \\
7 & 180 & $(56,123)$  & UNSAT & 25.8   & 18.3   \\
\end{tabular}
\label{tbl:heuristic-results}
\end{table}

\subsection{Enumeration of GR($k$) walks for $k\leq 6$}\label{sec:enumeration-kleq6}

We now describe our enumeration of all GR($k$) walks up to isomorphism for $k\leq6$.
For this, we use the basic CNF encoding described in Section~\ref{sec:encoding}
along with the unreachable point encoding of Section~\ref{sec:reachability}
(but not the constraint-removal heuristic).

For a fixed $k$, an incremental approach was used
to enumerate all GR($k$) walks.  A variable~$m$ was used to control
the number of steps in the walk.  Given $k$ and $m$, the enumeration
of all $m$-step GR($k$) walks was accomplished by generating
$m+1$ SAT instances, one for each ending point $(x,m-x)$ with
$x\in\{0,1,\dotsc,m\}$.  For each such point, a version of {\CaDiCaL} that exhaustively
finds \emph{all} solutions of a SAT instance
was used to
find all GR($k$) walks ending at the point $(x,m-x)$.
Once all GR($k$) walks with $m$ steps were known they were filtered
up to isomorphism using the equivalence operations described in Section~\ref{sec:symbreak}.

The equivalence filtering was done by converting each GR($k$) walk into a normal form
in such a way that all equivalent walks produce the same normal form.  To do this, each
$m$-step walk is represented as a binary string of length $m$, with \1 representing a north
step and \0 representing an east step.  The normal form of a GR($k$) walk is the
walk whose binary string is the lexicographically greatest of all walks in the same
equivalence class.

Once all $m$-step GR($k$) walks had been determined, if there was at least one
$m$-step GR($k$) walk then $m$ was incremented by 1 and the enumeration process was repeated.
Eventually, no $m$-step GR($k$) walks were found.  Once this happens, we have that
$a(k)=m$ and the length of the maximal GR($k$) walk(s) is $m-1$.
For example, we found that $a(4)=9$, $a(5)=29$, and $a(6)=97$, confirming
the results of Shallit~\cite{A231255}.
The computations for $k=4$ and $k=5$ completed in under a second of CPU time,
while the $k=6$ case required $118{,}990$ seconds.

Visual heatmap diagrams depicting our GR($k$) walk enumeration results for $k=4$ to $k=6$
are given in Figure~\ref{fig:heatmap}.  These diagrams visually show how many GR($k$)
walks (in normal form) exist from the origin to a point $(x,y)$ in the plane.
Only walks in normal form are counted, so it is possible
a point has walks to it when both the point below and the point to the left do not.
For example, $(9,4)$ is reachable using a GR(5) walk, but $(9,3)$
and $(8,4)$ are not reachable by walks in normal form.  $(8,4)$ is GR(5)-reachable, but only using a walk
which in normal form ends at $(4,8)$.

Once the value of $a(k)=m$ has been determined, we produce unsatisfiability certificates demonstrating the nonexistence
of $m$-step GR($k$) walks.\footnote{The certificates are available at \url{https://doi.org/10.5281/zenodo.17645678}.}
When combined with the known GR($k$) walks of length $m-1$ (which are easy to check for correctness)
this provides a certificate that $a(k)=m$.
The unsatisfiability certificates are in the DRAT format~\cite{drat}, a standard format for unsatisfiability proofs in modern SAT solving.
A DRAT proof consists of a list of clause additions and deletions.
Added clauses are redundant with respect to the current formula, meaning they may be added without changing the formula's satisfiability.
Deletions discard clauses when they are deemed to not be useful (in order to keep the memory footprint
of the solver manageable).  The last added clause in the DRAT proof is the empty clause.  The empty clause having
the same satisfiability status as the original formula proves
the original formula to be unsatisfiable, since no truth assignments satisfy the empty clause.
The proofs were validated with the proof verifier \textsc{DRAT-trim}~\cite{drattrim}.
This tool verifies that each clause addition is indeed a logical consequence of previously derived clauses.
Thus, the SAT solver itself does not need to be trusted; only the proof verifier---a much simpler
piece of software---needs to be trusted.

\begin{figure}
  \centering
    \begin{subfigure}{0.2225\linewidth}
        \includegraphics[width=\linewidth]{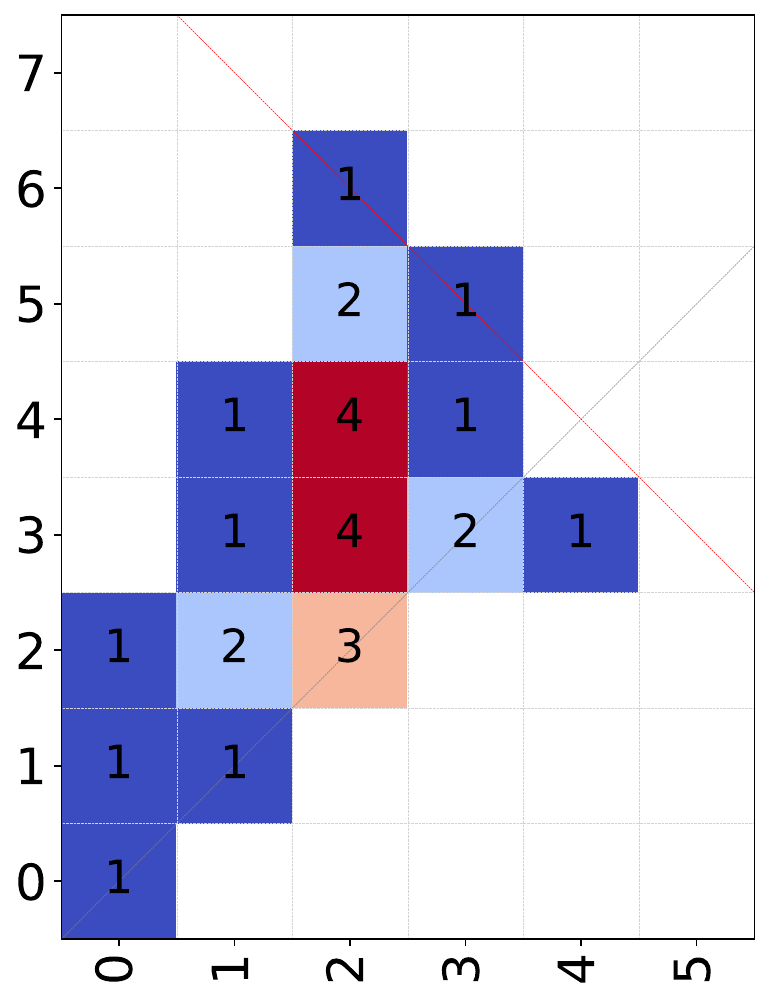}
        \caption{$k=4$}
  \end{subfigure}
  \begin{subfigure}{0.29\linewidth}
        \includegraphics[width=\linewidth]{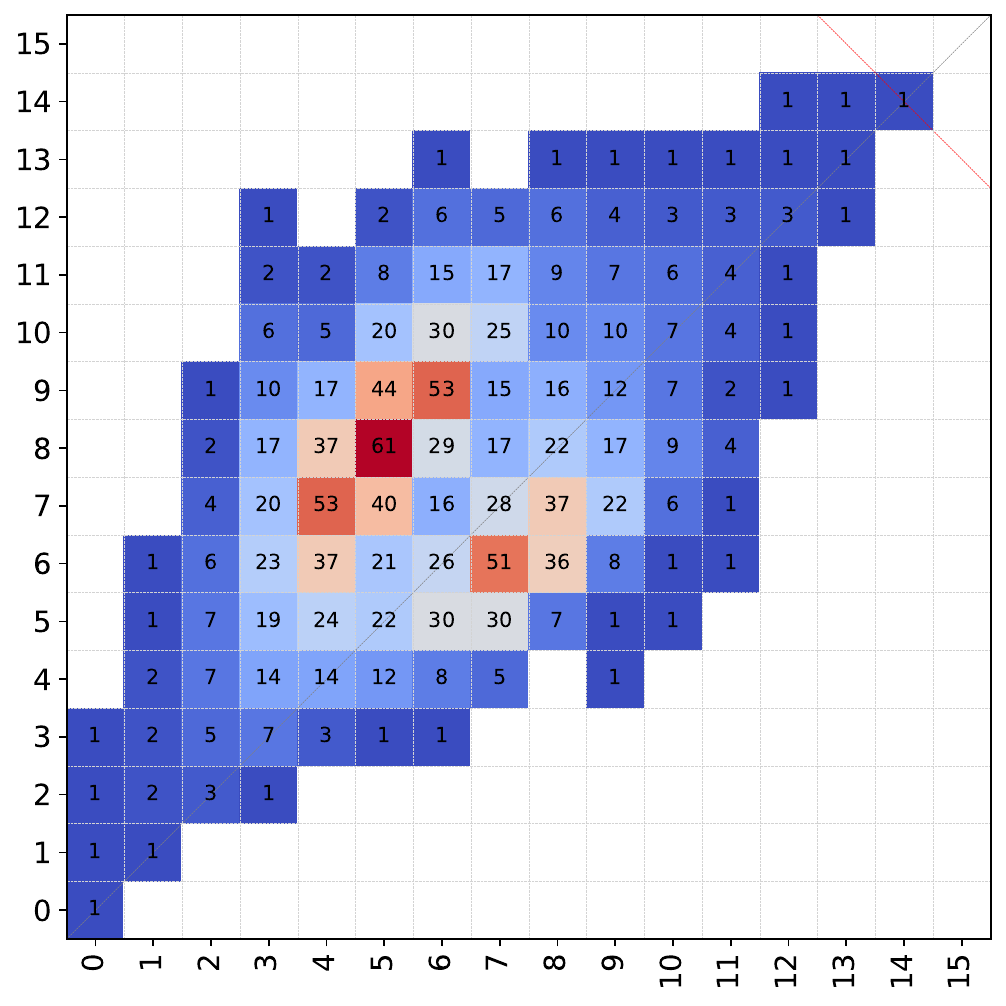}
        \caption{$k=5$}
  \end{subfigure}
  \begin{subfigure}{0.75\linewidth}
    \includegraphics[width=\linewidth]{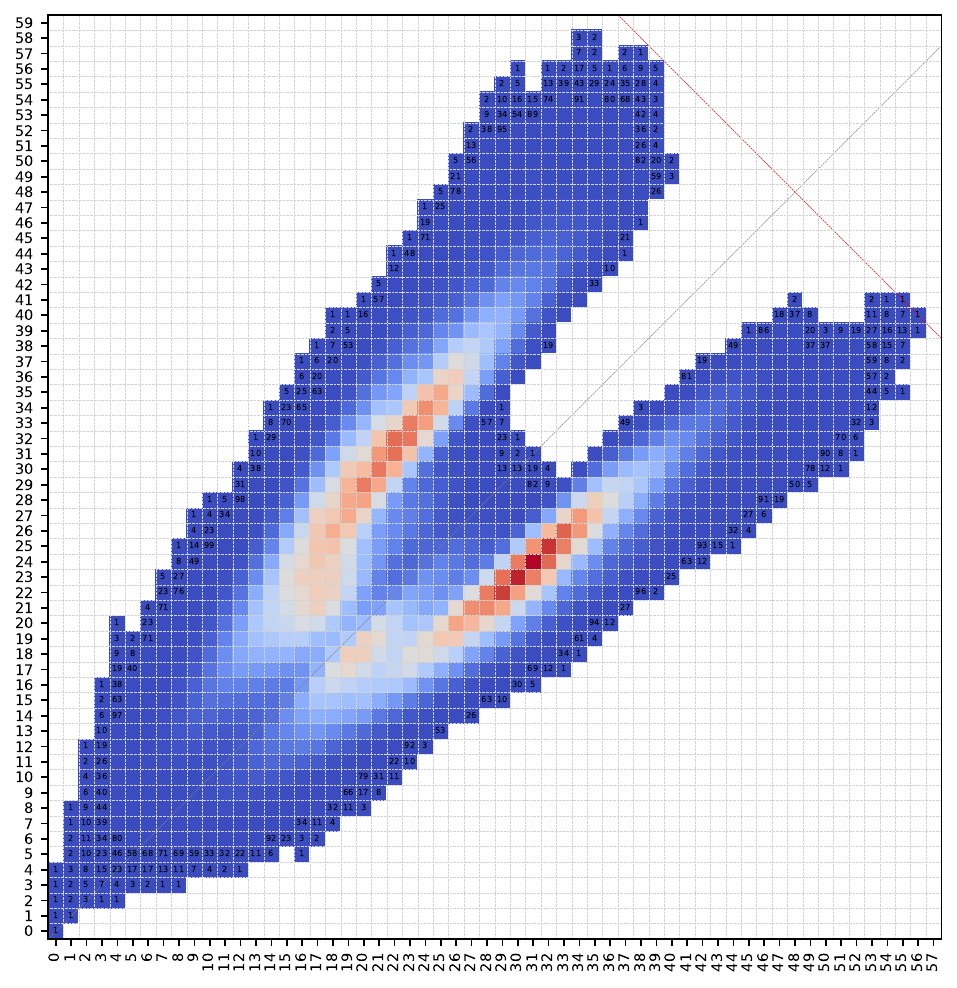}
    \caption{$k=6$}\label{fig:sub-heatmap}
  \end{subfigure}\hfill

  \caption{Exhaustive enumeration of GR($k$) walks for $k=4$ to $k=6$
  in normal form.
  Color intensity indicates the number of distinct walks reaching each point,
  with red indicating the greatest number of paths
  (the deepest red for $k=6$ corresponding to $324{,}571$ walks),
  blue indicating the least number of paths, and white indicating
  the point is unreachable using a GR($k$) walk in normal form.
  The red antidiagonal line corresponds to $y=(a(k)-1)-x$.
  }
  \label{fig:heatmap}
\end{figure}

\begin{table}
\caption{Results for proving $a(k)\leq m$ for $3 \leq k \leq 6$.
{\CaDiCaL} was used for solving and proof generation and \textsc{DRAT-trim} was used for proof verification.
Solving and verification times are given in seconds.
}
\centering
\begin{adjustbox}{max width=\textwidth}
\begin{tabular}{c c c c c}
$k$ & $m$ & Solve time & Proof size & Verification time \\
\hline
3 & 4  & $<1$ sec   & $<1$ KiB   & $<1$ sec \\
4 & 9  & $<1$ sec   & $<1$ KiB   & $<1$ sec \\
5 & 29 & $<1$ sec   & \phantom023 KiB     & $<1$ sec \\
6 & 97 & 311 sec    & 583 MiB    & 516 sec  \\
\end{tabular}
\end{adjustbox}
\label{tab:k6res}
\end{table}

Table~\ref{tab:k6res} summarizes the running times of the proof generation
and proof verification steps
for $3 \leq k \leq 6$, and these were done on the Ryzen machine.
The proof that $a(6)\leq97$ was generated by {\CaDiCaL} in 311 seconds and
was verified by \textsc{DRAT-trim} in 516 seconds.
This nonexistence certificate provides more trust
when compared to a traditional search program---because
a bug in a search program could cause GR(6) walks to be missed,
and there is no way to tell after-the-fact if there are no certificates that can be examined
for correctness.  However, for the purposes of double-checking and runtime comparison, we
wrote a custom backtracking search program in C++ that we used to search for
GR(6) walks of length 97, and this search took 2344 seconds on the Ryzen machine to confirm
that there are no GR(6) walks of length 97 (i.e., 7.5$\times$ slower than the SAT solver).
Although with more work the speed of the backtracking search program could likely be improved,
this is an indication that not only are SAT solvers more trustworthy
than custom search, they can also be more efficient.

\subsection{Reachability bounds for GR(7) walks}\label{sec:gr7reachability}

As explained in Section~\ref{sec:reachability}, if a point $(x,y)$ is known to be
GR($k$)-unreachable then we can add clauses encoding the unreachability of $(x,y)$,
and such clauses were shown to be helpful in Section~\ref{sec:reachability-performance}.
In Section~\ref{sec:enumeration-kleq6}, we determined all GR($k$)-reachable points
for $k\leq6$.  For $k=7$, the difficulty of the problem prevented us from determining
all GR(7)-reachable points.  However, we were able to determine upper and lower
reachability bounds for GR(7) walks of up to 267 steps.

\begin{figure}
    \centering
    \includegraphics[width=\linewidth]{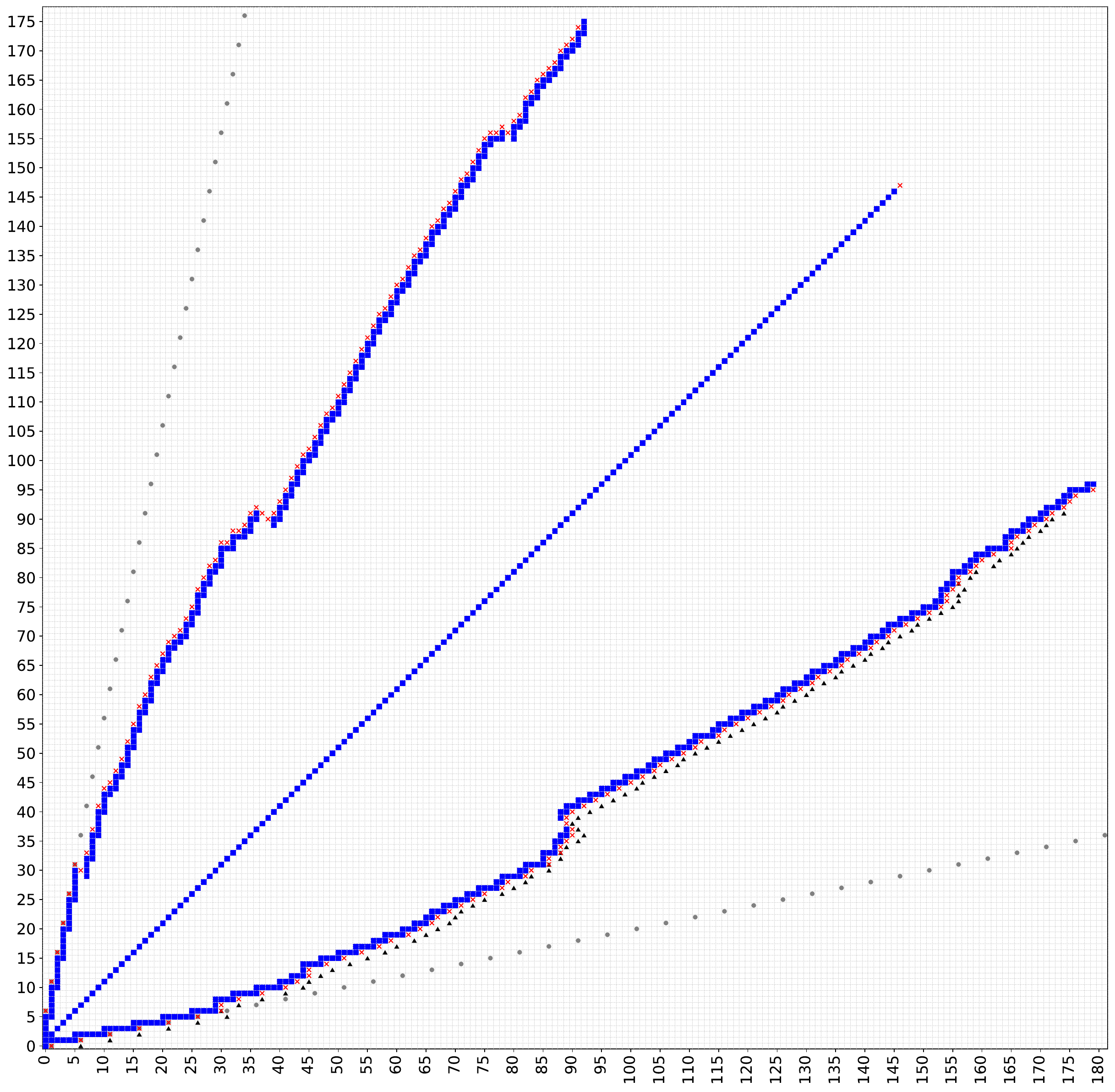}
    \caption{A partial reachability diagram for $k=7$.
    Blue squares are GR(7)-reachable points, and red crosses are GR(7)-unreachable points
    from the origin using a north first step.
    The black triangles show the reflected upper unreachability bound, and gray circles show the extremal bounds from Section~\ref{sec:extremalpoints}.
    }
    \label{fig:bounds7}
\end{figure}

A diagram showing the upper and lower bounds we found is provided in Figure~\ref{fig:bounds7}. 
These bounds were computed incrementally starting from the origin and using no symmetry breaking
except assuming the first step is north.  For the upper bound,
when a point $(x,y)$ was found to be reachable the point to solve was updated to
$(x,y+1)$ and the process restarted.  Conversely, if the point $(x,y)$
was found to be unreachable the point to solve was updated to $(x+1,y-1)$.
For the lower bound, when a point
$(x,y)$ was found to be reachable the next point to solve was set to $(x+1,y)$,
and if the point $(x,y)$ was found to be unreachable the next point to solve was set to
$(x-1,y+1)$.

In order to determine the upper and lower GR(7)-reachability bounds, we must
solve both satisfiable and unsatisfiable instances.  Furthermore, it is not
known in advance which instances are satisfiable and which are unsatisfiable.
A priori, it is not clear whether to use the CNF or KNF encoding, so we tried both
CNF and KNF, as well as a third ``hybrid'' mode of Cardinality-{\CaDiCaL}\@.  The hybrid mode switches between
solving with clauses and klauses, spending approximately half the time doing
traditional Boolean constraint propagation (i.e., operating on clauses)
and the other half using special cardinality-based propagation (i.e., operating
on klauses).  For the purposes of benchmarking,
we solved all satisfiable and unsatisfiable reachability instances on the
upper and lower boundary corresponding to GR(7) walks with $n\leq180$ points.
Each instance was solved 15 times,
using 15 random seeds for each instance, and the results are given in Table~\ref{tbl:cnfknfUB}.
The constraint-removal heuristic was not used for benchmarking, but it was used to determine the points in Figure~\ref{fig:bounds7}.  In twelve instances the satisfying assignment produced an invalid GR($k$) walk
due to the constraint removal heuristic.  In each of these cases the solver was rerun and a
valid GR($k$) walk was found on the next run of the solver.

\begin{table}
\caption{Performance of the CNF ({\CaDiCaL}), KNF (Cardinality-{\CaDiCaL} with {\tt ccdclMode=0}),
and Hybrid (Cardinality-{\CaDiCaL} with {\tt ccdclMode=1}) modes for computing all reachable (SAT)
and unreachable (UNSAT) boundary points across instances with $n\leq180$.
Times are reported as the total solve times (in seconds),
with each instance solved 15 times (each time with a different random seed),
and the median solve time used in the total.
}
\begin{adjustbox}{max width=\textwidth, center}
\begin{tabular}{c S[table-format=4.1] S[table-format=5.1] S[table-format=5.1]}
{Type} & {CNF}   & {KNF}     & {Hybrid} \\
\hline
SAT    &  2601.6 &   2261.6  &   8151.3  \\
UNSAT  &  4268.5 &  35295.1  &   7063.0  \\
Total  &  6870.1 &  37556.7  &  15214.2  \\
\end{tabular}
\end{adjustbox}%
\label{tbl:cnfknfUB}
\end{table}

The results show that KNF performed better on satisfiable instances and CNF
performed better on unsatisfiable instances.  However, the satisfiable
instances were unknown in advance, so CNF was a better choice
overall, given that KNF performed poorly on unsatisfiable instances and CNF outperformed the hybrid mode
on both satisfiable and unsatisfiable instances.  Thus,
the reachability of the upper and lower bounds in Figure~\ref{fig:bounds7}
was computed using the CNF encoding.  Overall, it took 
375.05 CPU days
to determine the reachability of the upper and lower boundaries in Figure~\ref{fig:bounds7}.
The most difficult single instance that was successfully solved
was the point $(176,94)$.
The solver took 63.5 days
and used up to 5.5~GiB of memory to determine $(176,94)$ was GR(7)-unreachable.

Given the shape of the GR(6)-reachability diagram in \Cref{fig:sub-heatmap},
we suspect there are GR(7)-unreachable points on and around the midline $y=x$ taking significantly
fewer than $a(7)$ steps to reach.  We made an effort to find such unreachable points by
using {\CaDiCaL} to determine the GR(7)-reachability of points along the line $y=x+1$.
Without using parallelization, the last point we were able to successfully determine the
reachability of was $(138,139)$.  A GR(7) walk to $(138,139)$ was found in
10.4 hours using the CNF encoding.
The CNF encoding was able to solve instances for larger $n$
than the KNF encoding;
perhaps an indication that the instances are becoming ``closer'' to unsatisfiable
in the sense that fewer satisfying assignments are present.  Using the KNF encoding,
the solver was unable to determine the reachability of points with $n>270$ and
on the line $y=x+1$ using a week of compute time. 

We incorporated parallelization in order to solve midline reachability instances with $n\geq280$.
We split the SAT instances into subinstances by enumerating all
points on lines of the form $y = sn - x$,
where~$s$ is chosen to be a simple slope such as $\frac{1}{2}$.
Each subinstance fixes one point on the line to true and the remaining
points on the line to false.
Using this approach, the instance with $n=294$ corresponding to the final point $(146,147)$
was determined to be GR(7)-unreachable (assuming an initial north step).
This computation required a two-stage splitting process.
The instance was first split by selecting all possible ways
of selecting a point $(x_0,y_0)$ on the line $y=\frac{1}{3}n-x$,
and a point $(x_1,y_1)$ on the line $y=\frac{2}{3}n-x$.
Subinstances leading to trivially unsatisfiable instances were disregarded, such as
when $(x_0,y_0)$ was GR(7)-unreachable, or when $(x_1,y_1)$ was unreachable from $(x_0,y_0)$.
This resulted in 1752 subinstances, corresponding to all simultaneously feasible
choices of $(x_0,y_0)$ and $(x_1,y_1)$.
Each of the 1752 subinstances ran for 7 days on the Nibi cluster,
and afterwards 213 subinstances remained unresolved.
The unsolved instances were split a second time by selecting all
points $(x_2,y_2)$ on the line $y=\frac{1}{2}n-x$ that were GR(7)-reachable from
both $(x_0,y_0)$ and $(x_1,y_1)$ simultaneously.
The second split produced 4554 subinstances, all of which were determined to be unsatisfiable in under 7 days.
The total running time across all subinstances was about $18.0$ CPU years.

Note the final GR(6)-reachable point on the line $y=x+1$
is $(30,31)$, and north--east walks to this point contain $n=62$ points.
The maximal GR(6) walk contains $n=97$ points,
so $a(6)$ is about 56\% larger than the number of points on the longest
GR(6) walk ending on the line $y=x+1$.  If this was also the case
for $k=7$, we would expect $a(7)$ to be around $457$.

\subsection{Searching for long GR(7) walks}\label{sec:gr7improvement}

The difficulty of the SAT instances in the Gerver--Ramsey collinearity problem with $k=7$
prevented us from determining the exact value of $a(7)$.  The SAT instances asserting
the existence of a GR(7) walk with $n$ points were feasible to solve without parallelization
up to around $n=300$.  In order to go farther, we tried several strategies of incorporating
parallelization.

The simplest parallelization strategy is simply to run multiple independent copies of the
solver on the same instance, with the only change being the random seed passed to Cardinality-{\CaDiCaL}\@.  Setting different
random seeds prevents the solver from making the same choices each time, permitting different parts
of the search space to be explored.  For the GR(7) instance with $n=305$ points, we used 150 random seeds
and ran each instance of the solver for three days.  Twenty-six of the 150 solver instances
found 305-point GR(7) walks, and all GR(7) walks found were distinct.

Ultimately, we had better results employing the cube-and-conquer parallelization strategy
described in Section~\ref{sec:c&c} and creating cubes with the lookahead solver {\march}.
The cube-and-conquer paradigm is typically used on unsatisfiable
SAT instances.  However, in our case we are aiming to find long
GR(7) walks, and therefore are looking to find satisfying assignments for instances for as large
$n$ as possible.  Thus, even if {\march} does its job of partitioning the SAT instance into
subinstances that are roughly of equal difficulty, there is no guarantee that this partition
will be useful for finding satisfying assignments.
For example, {\march} could yield subinstances for which all but one are unsatisfiable---this would limit
the benefit of using cube-and-conquer if the goal is to find a satisfying assignment.
We generated 150 cubes using {\march}'s
default parameters and none of the subinstances were found to be satisfiable after 3 days.
Examining the cubes produced, most of the literals in the cubes corresponded to points
that were close to the lower boundary and consequently most subinstances focused on searching
for GR(7) walks close to the lower boundary---not optimal for finding long GR(7) walks.

By modifying {\march} to branch on variables corresponding to points around $(70,70)$,
we computed two other sets of 150 cubes each.
Altogether, including the
GR(7) walks found by the random seed approach and a fourth set
of cubes described below, we found 125 distinct
305-point GR(7) walks (discarding two satisfying assignments
that were invalid due to the constraint removal heuristic).
A heatmap of the points visited by these walks is
provided in Figure~\ref{fig:n305_heatmap}.
Streamlining unit clauses were included in these instances to enforce that the path
did not stray more than 25 points diagonally from the line $y=x+1$.
This streamlining was added after preliminary experimentation
found twenty GR(7) walks with $n\geq310$ points and
all these walks did not stray more than 25 points from $y=x+1$.

\begin{figure}
    \centering
    \includegraphics[width=0.5\linewidth]{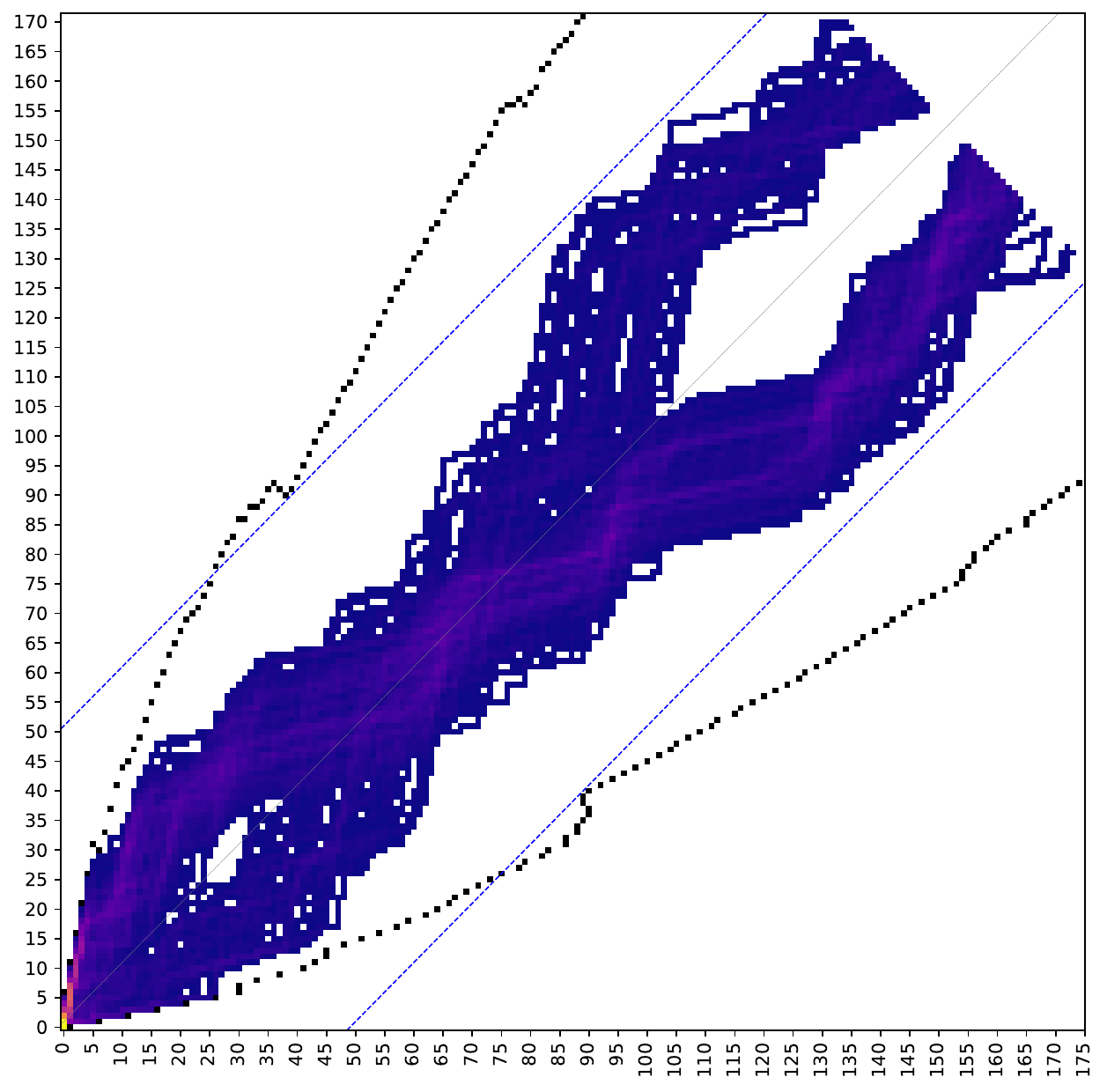}
    \caption{Heatmap of points visited by 125 distinct 305-point walks found using parallelization.
    Also included in the plot is the midline $y=x+1$ as well as the lines $y=x+1\pm50$.
    } 
    \label{fig:n305_heatmap}
\end{figure}

A fourth set of cubes ultimately led to the longest GR(7)
walk we found.  They were produced using a SAT instance not containing the streamlining unit clauses and
with some boundary points mistakenly
encoded incorrectly.  Some cubes became trivially unsatisfiable once the instance was
corrected and the streamlining unit clauses were added, leaving 133 useful cubes.
Despite this, these cubes ultimately resulted in the longest GR(7) walks we found.
A summary of the GR(7) walks found with this set of cubes is
provided in Table~\ref{tbl:k7results}.
The two distinct solutions
for $n=323$ are visually shown in Figure~\ref{fig:solutions_323}.
Counting the total time spent across all 133 cubes,
these solutions took 346.0 CPU days to find.

\begin{table}
\caption{
Summary of the GR(7) walks found using the best performing
set of cubes with the streamlining technique.
There were 133 non-trivial cubes and each cube was used
to produce a KNF instance on which Cardinality-{\CaDiCaL} was run for 72 hours.
Times are in seconds.}
\centering
\begin{adjustbox}{max width=\textwidth}
\begin{tabular}{c c c c c}
$n$ & Solutions & Min & Max & Median  \\
\hline
323 & 2  & 43762.2    & 59735.7 & 51748.9 \\
317 & 7     & 7962.2 & 88252.3 & 54885.6 \\
311 & 19   & 2100.1 & 90337.3 & 53017.2 \\
305 & 38   & 334.3 & 104144.6 & 51792.3 \\
\end{tabular}
\end{adjustbox}
\label{tbl:k7results}
\end{table}

\begin{figure}
  \centering
    \begin{subfigure}{0.49\linewidth}
        \includegraphics[width=\linewidth]{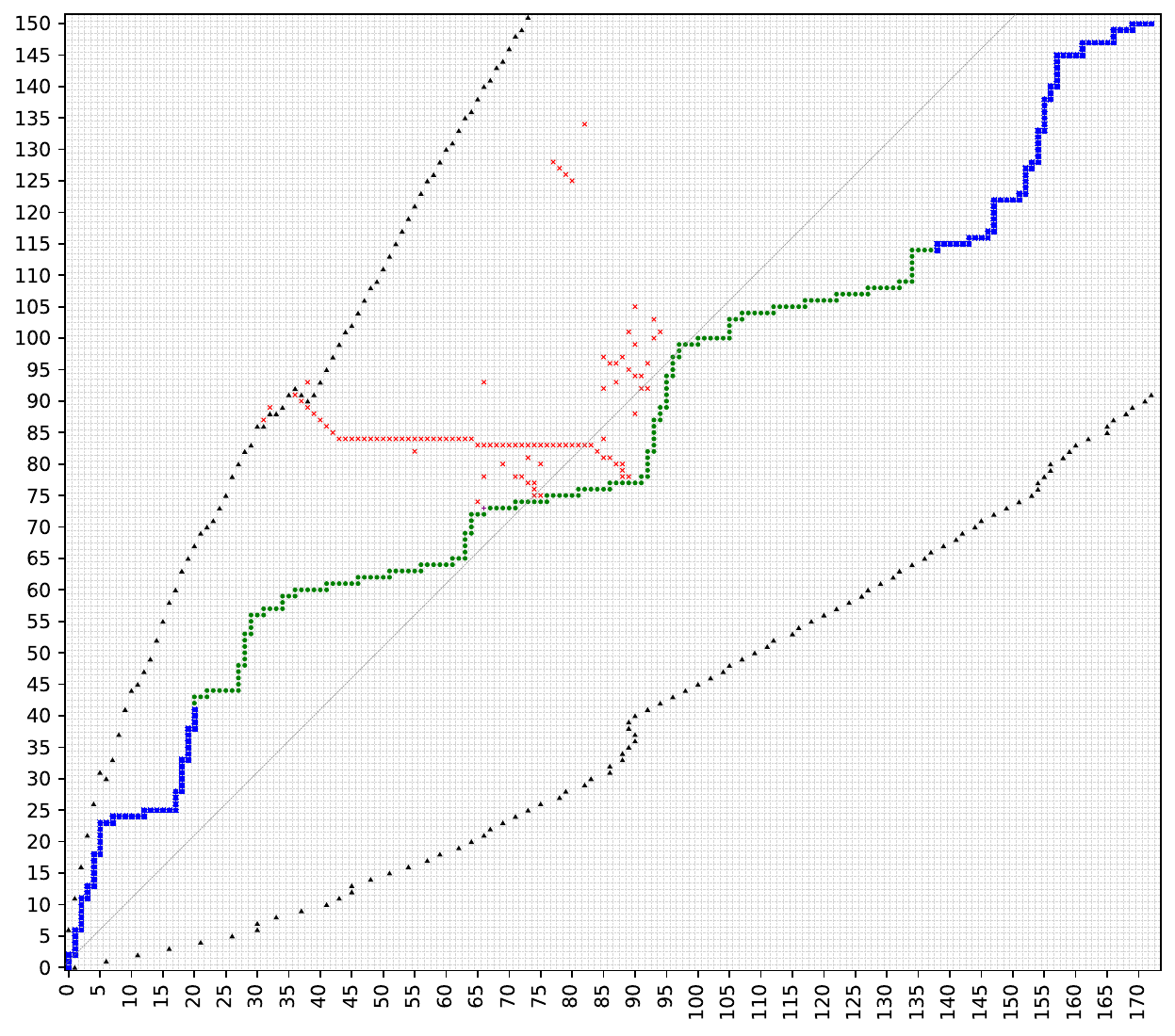}
  \end{subfigure}
  \begin{subfigure}{0.49\linewidth}
        \includegraphics[width=\linewidth]{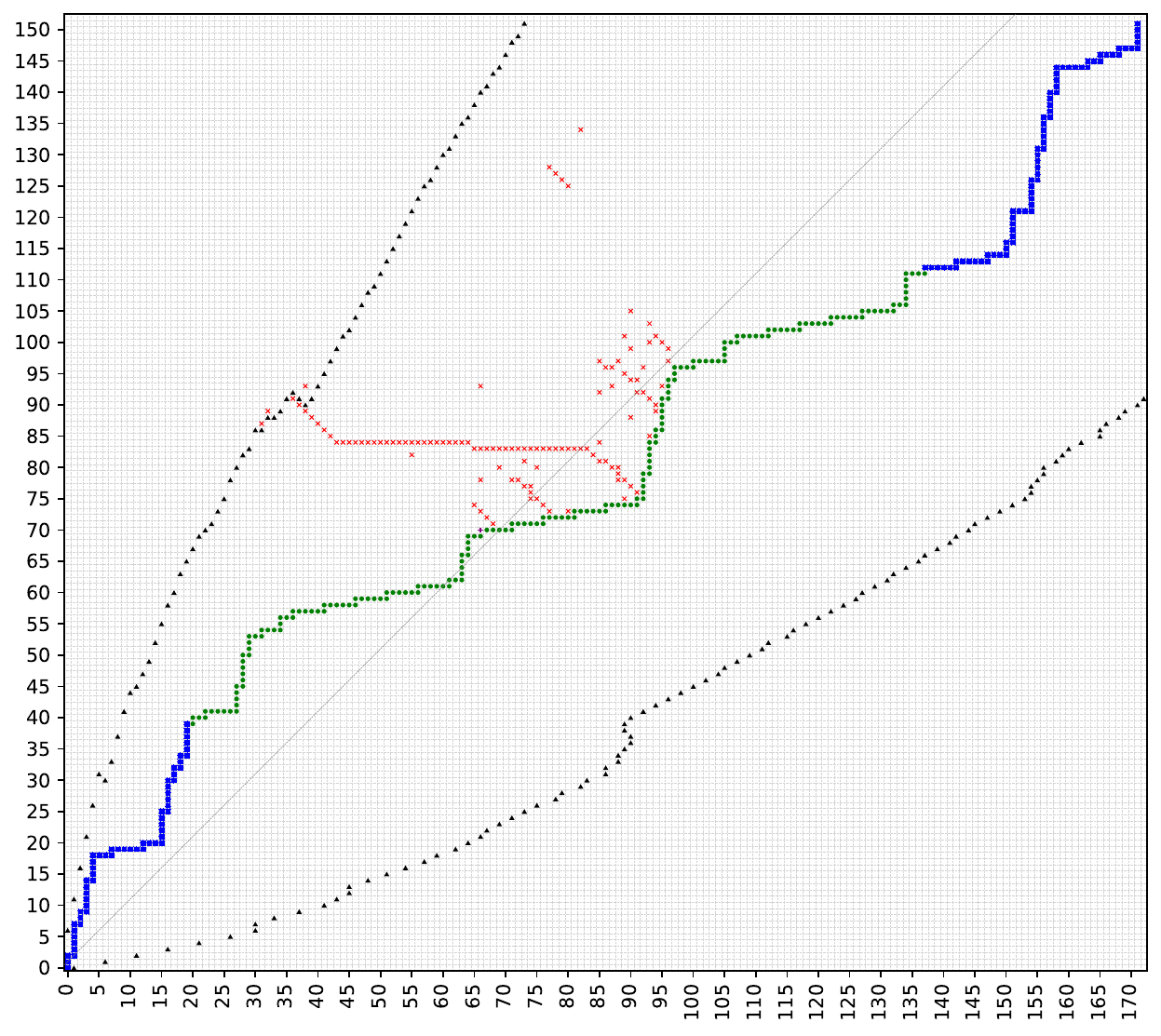}
  \end{subfigure}
\hfill
  \caption{Two 323-point GR(7) walks, with their associated cubes. The positive literal in the cube is shown as a purple plus marker, and the negative literals as red crosses.
  The upper and lower GR(7) unreachable boundary points are shown as black triangles.
  The path points shown as green circles are common to both GR(7) walks. 
  }
  \label{fig:solutions_323}
\end{figure}

An examination of the two 323-point GR(7) walks found using cube-and-conquer revealed that a
188-step component was shared between the two walks.  Given this, we ran additional searches to see if the common subpath
could be extended to GR(7) walks with more than 323 points.
Ultimately, up to isomorphism we found ten 328-point GR(7) walks containing this subpath
or its complement via an exhaustive search with {\CaDiCaL} in 40,931 seconds.
Using 6.6 CPU days on the Ryzen desktop computer, {\CaDiCaL} was also able to show that no 329-point GR(7) walk exists containing the common subpath or its complement,
even with 20 points removed from both ends of the common subpath.

Finding extensions of a given subpath was accomplished using an extension of our SAT encoding.
In addition to the variables and constraints given in Section~\ref{sec:encoding}, we also add new variables
$r_i$ representing that the $i$th step of the walk is east (where $0\leq i<m$).
If both $(x,y)$ and $(x+1,y)$ are on the path, the $(x+y)$th step
is eastward.  Similarly, if $(x+1,y)$ is on the path and the $(x+y)$th step
was east, the previous point was $(x,y)$.  Thus, we use the clauses
\[
(v_{x,y} \land v_{x+1,y}) \limp r_{x+y} \;\text{ and }\; (v_{x+1,y} \land r_{x+y}) \limp v_{x,y} \qquad
\text{for all $x\geq0$, $y\geq0$, and $x+y<n-1$.}
\]
When $r_i$ is false, this represents that the $i$th step was north.  This case
is handled similarly, using the clauses
\[ (v_{x,y} \land v_{x,y+1}) \limp \lnot r_{x+y} \;\text{ and }\; (v_{x,y+1} \land \lnot r_{x+y}) \limp v_{x,y}
\qquad \text{for all $x\geq0$, $y\geq0$, and $x+y<n-1$.}
\]

Say that $B\coloneqq b_0b_1\ldots b_{\ell-1}$
is a binary string (with \1 denoting a north step) representing
the $\ell$-step path that we want to extend.
In other words,
the steps in $B$ must appear
as a subpath in every GR($k$) walk produced by a satisfying assignment.
We also introduce the variables $s_i$ representing that the subpath $B$ starts
immediately following the $i$th step (where $0\leq i<n-\ell$).  If the subpath $B$ starts
after the $i$th step, that means $r_{i+j}$ is true exactly when $b_j=\0$ for $0\leq j<\ell$.
Thus, we use the conjunction of clauses
\[
\bigwedge_{\substack{0\leq j<\ell\\ b_j=\0}} (s_i \limp r_{i+j})
\qquad\text{and}\qquad
\bigwedge_{\substack{0\leq j<\ell\\ b_j=\1}} (s_i \limp \lnot r_{i+j})
\qquad\text{for all $0\leq i<n-\ell$}
\]
to encode that the path $B$ appears after the $i$th step when $s_i$ is true.
Then we can enforce that the path $B$ appears somewhere in the GR($k$) walk
by the clause $s_0\lor s_1\lor\dotsb\lor s_{n-1-\ell}$, which says that the subpath $B$ must start
somewhere in the path.

This encoding determined that the
188-step common subpath in our two 323-point GR(7) walks can be extended
to GR(7) walks with up to 328 points, but no more. 
One of the 328-point GR(7) walks we found is visually shown in Figure~\ref{fig:328}. The binary sequence of steps in the walk is
\begin{gather*}
\text{\scriptsize\texttt{1100000100000100010000010001011111011111001000001000001001000001011011111000010000010001101110011011111011111}} \\[-1.5ex]
\text{\scriptsize\texttt{0111110111110111110110000100011011111011111011111011111011111010000100000100100000100010011101111100011011111}} \\[-1.5ex]
\text{\scriptsize\texttt{0111110111110111110111110110000011101111101111101110010000011100000100000100000100001000011111011011111000100}}\rlap{.}
\end{gather*}
By definition, there are necessarily no lines passing through 7 points on this walk,
but there are 196 lines passing through 6 points on the walk
and these lines are also drawn in Figure~\ref{fig:328}.

\begin{figure}
\begin{center}
\includegraphics[width=\textwidth]{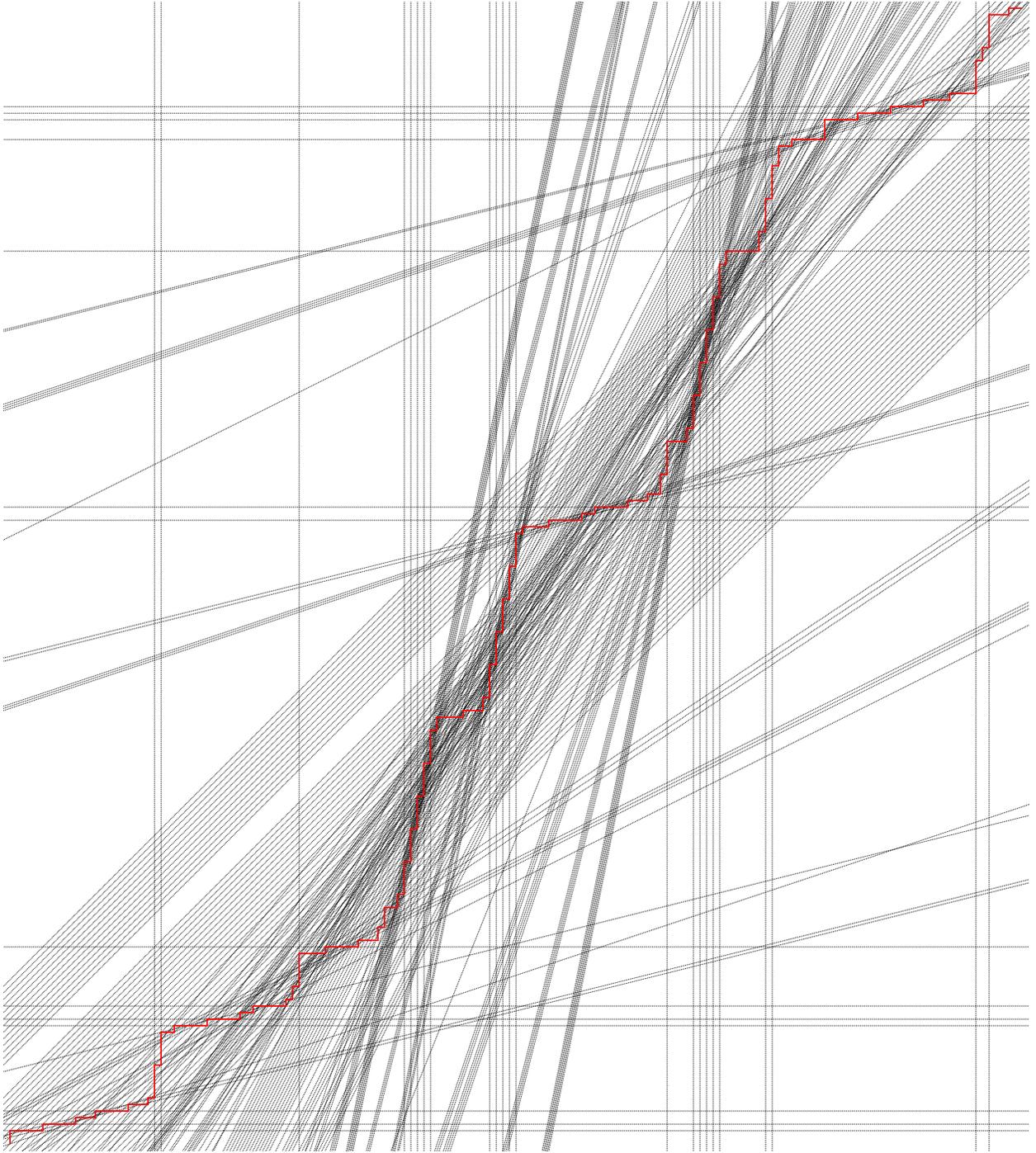}
\end{center}
\caption{A visual depiction of a 328-point GR(7) walk found by our approach, along with all 196 lines passing
through 6 lattice points on the walk.}
\label{fig:328}
\end{figure}

\section{Conclusion}\label{sec:conclusion}

In this paper we devised a satisfiability-based approach for studying the Gerver--Ramsey collinearity problem
on north--east lattice paths.  As a result of our work we enumerated all maximal GR($k$) walks for $k\leq6$
and made progress on the problem for $k=7$, improving the longest known GR(7) walk from 260 steps to 327 steps.
Although we were unsuccessful in finding a maximal GR(7) walk, we hope that the introduction of SAT solving
on this problem leads to more progress and ultimately the determination of the value of $a(7)$.

In addition to determining values of $a(k)$ for larger $k$, there are a number of
related problems that may be of interest.  One variant would be to generalize the allowed steps in the walk.
In this paper we have always assumed a step set of $\{(1,0),(0,1)\}$, but
Gerver and Ramsey show that regardless of the step set $S\subset\Zee^2$ and the value of $k$,
sufficiently long $S$-walks must always contain $k$ collinear points.

Another interesting variant would be to consider a three-dimensional variant of the problem.
Say $b(k)$ denotes the number of points in the longest lattice path with steps in
$\{(1,0,0),(0,1,0),(0,0,1)\}$ avoiding $k$ collinear points.  In contrast with the two-dimensional
case, $b(k)$ may be infinite.
Gerver and Ramsey note that $b(3)=9$, but they also prove that $b(5^{11}+1)=\infty$ by construction
of an infinite walk $W\subset\Enn^3$ having at most $5^{11}$ collinear points.
Moreover, they conjecture that $W$ actually has at most three collinear points,
which would imply $b(k)=\infty$ for all $k\geq4$.

Recently, Lidbetter~\cite{Lidbetter2024} determined their conjecture to be false by finding
six collinear points in $W$.  He also showed that $W$ does not contain 189 collinear points, and
as a consequence proves $b(k)=\infty$ for all $k\geq189$.
Even though Gerver and Ramsey's conjecture that $W$ avoids four collinear points
was wrong, it may still be the case that
$b(4)=\infty$ because another infinite walk may avoid four collinear points.
The 41st step of $W$ is the first creating
four collinear points, so $b(4)\geq41$.  We are not aware of a better bound on $b(4)$,
and perhaps a SAT approach could be used to improve this bound
or even determine $b(4)$, if it happened to be finite.

\section*{Acknowledgements}

We thank Joseph Reeves for answering a question about configuring Cardinality-{\CaDiCaL} and for
providing a KNF to CNF converter used in our work.
We also thank the anonymous referee for their useful comments which improved the paper.

\bibliographystyle{adamjoucc}
\bibliography{paper}

@article{GerverPaper,
  title={Long walks in the plane with few collinear points},
  author={Joseph L. Gerver},
  journal={Pacific Journal of Mathematics},
  year={1979},
  volume={83},
  number={2},
  pages={349--355},
  doi={10.2140/pjm.1979.83.349}
}

@article{GerverRamsey1979,
  author       = {Gerver, Joseph Leonide and Ramsey, Lawrence Thom},
  title        = {On Certain Sequences of Lattice Points},
  journal      = {Pacific Journal of Mathematics},
  volume       = {83},
  number       = {2},
  pages        = {357--363},
  year         = {1979},
  doi          = {10.2140/pjm.1979.83.357},
}

@InProceedings{Heule2012CubeAndConquer,
author="Heule, Marijn J. H.
and Kullmann, Oliver
and Wieringa, Siert
and Biere, Armin",
editor="Eder, Kerstin
and Louren{\c{c}}o, Jo{\~a}o
and Shehory, Onn",
title="Cube and Conquer: Guiding {CDCL} {SAT} Solvers by Lookaheads",
booktitle="Hardware and Software: Verification and Testing",
year="2012",
publisher="Springer Berlin Heidelberg",
address="Berlin, Heidelberg",
pages="50--65",
doi="10.1007/978-3-642-34188-5_8",
volume="7261",
series="Lecture Notes in Computer Science"
}

@InProceedings{totalizer,
author="Bailleux, Olivier
and Boufkhad, Yacine",
editor="Rossi, Francesca",
title="{Efficient {CNF} Encoding of Boolean Cardinality Constraints}",
booktitle="Principles and Practice of Constraint Programming -- CP 2003",
year="2003",
publisher="Springer Berlin Heidelberg",
address="Berlin, Heidelberg",
pages="108--122",
isbn="978-3-540-45193-8",
doi="10.1007/978-3-540-45193-8_8",
volume="2833",
series="Lecture Notes in Computer Science"
}

@inproceedings{Sinz2005,
  title = {Towards an Optimal {CNF} Encoding of Boolean Cardinality Constraints},
  ISBN = {9783540320500},
  ISSN = {1611-3349},
  DOI = {10.1007/11564751_73},
  booktitle = {Principles and Practice of Constraint Programming - CP 2005},
  publisher = {Springer Berlin Heidelberg},
  author = {Sinz,  Carsten},
  editor = {van Beek, P.},
  year = {2005},
  pages = {827–831},
  series={Lecture Notes in Computer Science},
  volume={3709}
}

@inproceedings{CadicalPaper,
title={{CaDiCaL} 2.0},
DOI={10.1007/978-3-031-65627-9_7},
author={Biere, Armin and Faller, Tobias and Fazekas, Katalin and Fleury, Mathias and Froleyks, Nils and Pollitt, Florian},
year={2024},
pages={133–152},
series={Lecture Notes in Computer Science},
booktitle={Computer Aided Verification -- CAV 2024},
publisher={Springer},
volume={14681},
address={Cham},
editor={Gurfinkel, A. and Ganesh, V.}
}

@inproceedings{LamPaper,
  title={A {SAT}-based resolution of {Lam}'s problem},
  volume={35},
  DOI={10.1609/aaai.v35i5.16483},
  number={5},
  journal={Proceedings of the AAAI Conference on Artificial Intelligence},
  author={Bright, Curtis and Cheung, Kevin K. H. and Stevens, Brett and Kotsireas, Ilias and Ganesh, Vijay},
  year={2021},
  month={May},
  pages={3669–3676}
}

@inproceedings{CardinalityCadical,
author = {Reeves, Joseph E. and Heule, Marijn J. H. and Bryant, Randal E.},
title = {From Clauses to Klauses},
year = {2024},
isbn = {978-3-031-65626-2},
publisher = {Springer-Verlag},
address = {Berlin, Heidelberg},
editor = {Gurfinkel, A. and Ganesh, V.},
doi = {10.1007/978-3-031-65627-9_6},
booktitle = {Computer Aided Verification: 36th International Conference, CAV 2024, Montreal, QC, Canada, July 24–27, 2024, Proceedings, Part I},
pages = {110–132},
numpages = {23},
location = {Montreal, QC, Canada},
volume={14681},
series={Lecture Notes in Computer Science}
}

@article{Brown1971,
  title={Advanced Problem 5811},
  journal={The American Mathematical Monthly},
  author={Tom C. Brown},
  year={1971},
  volume={78},
  number={7},
  pages={798},
  doi={10.1080/00029890.1971.11992858}
}

@article{Montgomery1972,
  title={Collinear Points on a Monotonic Polygon},
  journal={The American Mathematical Monthly},
  author={P. L. Montgomery},
  year={1972},
  volume={79},
  number={10},
  pages={1143--1144},
  doi={10.1080/00029890.1972.11993206}
}

@misc{A231255,
  title={On-Line Encyclopedia of Integer Sequences entry {A231255}},
  author={Jeff Shallit},
  year={2013},
  howpublished={\url{https://oeis.org/A231255}}
}

@inproceedings{Bright2020,
  title = {Effective Problem Solving Using {SAT} Solvers},
  ISBN = {9783030412586},
  ISSN = {1865-0937},
  DOI = {10.1007/978-3-030-41258-6_15},
  booktitle = {Maple in Mathematics Education and Research},
  publisher = {Springer International Publishing},
  series={Communications in Computer and Information Science},
  volume={1125},
  author = {Bright,  Curtis and Gerhard,  J\"{u}rgen and Kotsireas,  Ilias and Ganesh,  Vijay},
  year = {2020},
  pages = {205–219},
  editor={Gerhard, J. and Kotsireas, I.},
  address={Cham}
}

@inproceedings{Heule2024,
  title = {Happy Ending: An Empty Hexagon in Every Set of 30 Points},
  ISBN = {9783031572463},
  ISSN = {1611-3349},
  DOI = {10.1007/978-3-031-57246-3_5},
  booktitle = {30th International Conference on Tools and Algorithms for the Construction and Analysis of Systems},
  publisher = {Springer Nature Switzerland},
  author = {Heule,  Marijn J. H. and Scheucher,  Manfred},
  year = {2024},
  pages = {61–80},
  volume= {14570},
  series={Lecture Notes in Computer Science},
  editor={Finkbeiner, B. and Kovács, L.},
  address={Cham}
}

@inproceedings{Subercaseaux2025,
  title = {Automated Symmetric Constructions in Discrete Geometry},
  ISBN = {9783032070210},
  ISSN = {1611-3349},
  DOI = {10.1007/978-3-032-07021-0_3},
  booktitle = {Intelligent Computer Mathematics},
  publisher = {Springer Nature Switzerland},
  author = {Subercaseaux,  Bernardo and Mackey,  Ethan and Qian,  Long and Heule,  Marijn},
  year = {2025},
  month = oct,
  pages = {29–47},
  series={Lecture Notes in Computer Science},
  volume={16136},
  editor={de Paiva, V. and Koepke, P.}
}

@inproceedings{Heule2017,
  series = {IJCAI-2017},
  title = {Solving Very Hard Problems: Cube-and-Conquer, a Hybrid {SAT} Solving Method},
  DOI = {10.24963/ijcai.2017/683},
  booktitle = {Proceedings of the Twenty-Sixth International Joint Conference on Artificial Intelligence},
  publisher = {International Joint Conferences on Artificial Intelligence Organization},
  author = {Heule,  Marijn J. H. and Kullmann,  Oliver and Marek,  Victor W.},
  year = {2017},
  month = aug,
  pages = {4864–4868},
  collection = {IJCAI-2017}
}

@article{Bright2022,
  title = {When satisfiability solving meets symbolic computation},
  volume = {65},
  ISSN = {1557-7317},
  DOI = {10.1145/3500921},
  number = {7},
  journal = {Communications of the ACM},
  publisher = {Association for Computing Machinery (ACM)},
  author = {Bright,  Curtis and Kotsireas,  Ilias and Ganesh,  Vijay},
  year = {2022},
  month = jun,
  pages = {64–72}
}

@inproceedings{Cook1971,
  series = {STOC '71},
  title = {The complexity of theorem-proving procedures},
  DOI = {10.1145/800157.805047},
  booktitle = {Proceedings of the third annual ACM symposium on Theory of computing  - STOC ’71},
  publisher = {ACM Press},
  author = {Cook,  Stephen A.},
  year = {1971},
  pages = {151–158},
  collection = {STOC ’71}
}

@InProceedings{drattrim,
author="Wetzler, Nathan
and Heule, Marijn J. H.
and Hunt, Warren A.",
editor="Sinz, Carsten
and Egly, Uwe",
title="{DRAT}-trim: Efficient Checking and Trimming Using Expressive Clausal Proofs",
booktitle="Theory and Applications of Satisfiability Testing -- SAT 2014",
year="2014",
publisher="Springer International Publishing",
address="Cham",
pages="422--429",
doi="10.1007/978-3-319-09284-3_31",
series={Lecture Notes in Computer Science},
volume={8561}
}

@InProceedings{drat,
author="Buss, Sam
and Thapen, Neil",
editor="Janota, Mikol{\'a}{\v{s}} and Lynce, In{\^e}s",
title="{DRAT} Proofs, Propagation Redundancy, and Extended Resolution",
booktitle="Theory and Applications of Satisfiability Testing -- SAT 2019",
year="2019",
publisher="Springer International Publishing",
address="Cham",
pages="71--89",
doi="10.1007/978-3-030-24258-9_5",
series={Lecture Notes in Computer Science},
volume={11628}
}

@inproceedings{Heule2005,
  title = {March\_eq: Implementing Additional Reasoning into an Efficient Look-Ahead {SAT} Solver},
  ISBN = {9783540315803},
  ISSN = {1611-3349},
  DOI = {10.1007/11527695_26},
  booktitle = {Theory and Applications of Satisfiability Testing -- SAT 2004},
  publisher = {Springer Berlin Heidelberg},
  series={Lecture Notes in Computer Science},
  volume={3542},
  author = {Heule,  Marijn and Dufour,  Mark and van Zwieten,  Joris and van Maaren,  Hans},
  year = {2005},
  pages = {345–359},
  editor={Hoos, H.H. and Mitchell, D.G.}
}

@inproceedings{PySAT,
  author    = {Alexey Ignatiev and Antonio Morgado and Joao Marques{-}Silva},
  title     = {{PySAT:} {A} {Python} Toolkit for Prototyping with {SAT} Oracles},
  booktitle = {Theory and Applications of Satisfiability Testing -- SAT 2018},
  pages     = {428--437},
  year      = {2018},
  doi       = {10.1007/978-3-319-94144-8_26},
  series    = {Lecture Notes in Computer Science},
  volume    = {10929},
  publisher = {Springer},
  address   = {Cham},
  editor    = {Beyersdorff, Olaf and Wintersteiger, Christoph M.}
}

@article{crux,
  author={Michael W. Ecker},
  journal={Crux Mathematicorum},
  title={Problem 408},
  pages={294--296},
  year={1979},
  month={Dec},
  volume={10},
  publisher={Canadian Mathematical Society},
  url={https://cms.math.ca/wp-content/uploads/crux-pdfs/Crux_v5n10_Dec.pdf}
}

@article{Lidbetter2024,
  title = {Improved bound for the {Gerver-Ramsey} collinearity problem},
  volume = {347},
  ISSN = {0012-365X},
  DOI = {10.1016/j.disc.2023.113718},
  number = {1},
  journal = {Discrete Mathematics},
  publisher = {Elsevier BV},
  author = {Lidbetter,  Thomas F.},
  year = {2024},
  month = jan,
  pages = {113718}
}

@inproceedings{Subercaseaux2023,
  title = {The Packing Chromatic Number of the Infinite Square Grid is 15},
  ISBN = {9783031308239},
  ISSN = {1611-3349},
  DOI = {10.1007/978-3-031-30823-9_20},
  booktitle = {Tools and Algorithms for the Construction and Analysis of Systems},
  publisher = {Springer Nature Switzerland},
  author = {Subercaseaux,  Bernardo and Heule,  Marijn J. H.},
  year = {2023},
  pages = {389–406},
  volume={13993},
  series={Lecture Notes in Computer Science},
  editor={Sankaranarayanan, S. and Sharygina, N.},
  address={Cham}
}

@inproceedings{Subercaseaux2024,
  title = {Automated Mathematical Discovery and Verification: Minimizing Pentagons in the Plane},
  ISBN = {9783031669972},
  ISSN = {1611-3349},
  DOI = {10.1007/978-3-031-66997-2_2},
  booktitle = {Intelligent Computer Mathematics},
  publisher = {Springer Nature Switzerland},
  author = {Subercaseaux,  Bernardo and Mackey,  John and Heule,  Marijn J. H. and Martins,  Ruben},
  year = {2024},
  pages = {21–41},
  volume={14960},
  series={Lecture Notes in Computer Science},
  editor={Kohlhase, A. and Kovács, L.},
  address={Cham}
}

@phdthesis{Reeves2025,
title={Cardinality Constraints in Boolean Satisfiability Solving},
author={Joseph E. Reeves},
year={2025},
school={Carnegie Mellon University}
}

\end{document}